\documentclass[12pt]{article}

\usepackage{times}
\usepackage[centertags]{amsmath}
\usepackage{amssymb}
\usepackage{amsthm}
\usepackage[utf8]{inputenc}
\usepackage[T1]{fontenc}
\usepackage[a4paper, top=5cm, bottom=5cm, left=3.5cm, right=2.5cm]{geometry}
\usepackage[italian,english]{babel}
\usepackage{float}
\usepackage{graphicx}
\usepackage{tikz}
\usepackage{qtree}
\usepackage{tikz-qtree}
\usepackage[lined,algonl,boxed]{algorithm2e}
\usepackage[colorlinks=true,urlcolor=red,citecolor=red,linkcolor=red]{hyperref}
\usepackage{pgfplots}
\pgfplotsset{compat=1.13}
\usetikzlibrary{trees,positioning,shapes}
\usetikzlibrary{patterns}
\usepackage{verbatim}
\newtheorem{teo}{Theorem}
\newtheorem{defi}[teo]{Definition}
\newtheorem{prop}[teo]{Proposition}
\newtheorem{lem}[teo]{Lemma}
\newtheorem{cor}[teo]{Corollary}
\theoremstyle{definition}
\newtheorem{oss}[teo]{Remark}
\newtheorem{conj}[teo]{Conjecture}
\theoremstyle{definition}
\newtheorem{ex}[teo]{Example}

\newcommand{\bs}{\boldsymbol}
\newcommand{\N}{\mathbb{N}}
\newcommand{\Z}{\mathbb{Z}}

\newcommand{\K}{\mathbb{K}}

\newcommand{\sgen}[1]{\langle #1 \rangle_{\oplus}}

\newenvironment{acknowledgements}%
{\null\vfill\begin{center}%
		\bfseries Acknowledgements\end{center}}%
{\vfill\null}
\newcommand{\red}[1]{\langle\langle #1 \rangle\rangle}

\newcounter{lastnote}

\title{Embedding dimension of a good semigroup} 
\author{N. Maugeri, G. Zito}
\date{}

\begin{document}
	\linespread{1,0}
	
	\maketitle
	\begin{abstract}
		In this paper, we study good semigroups of $\N^n$, a class of semigroups that contains the value semigroups of algebroid curves with $n$ branches. We give the definition of embedding dimension of a good semigroup showing that, in the case of good semigroups of $\N^2$, some of its properties agree with the analogue concepts defined for numerical semigroups.

	\end{abstract}
	\section*{Introduction}

The concept of good subsemigroup of $\mathbb{N}^n$ was formally introduced in \cite{anal:unr}. Its definition arises from the properties of the value semigroups of one dimensional analytically unramified rings (for example the local rings of an algebraic curve) that  were initially studied in \cite{two:danna, gener:cdgz, symm:cdk, canonical:danna, multibranch:delgado, symm:delgado, semig:garcia}.
In \cite{anal:unr}, the authors proved that the class of good semigroups is actually larger than the one of value semigroups. Thus, such semigroups can be seen as a natural generalization of numerical semigroups and studied without necessarily referring to the ring theory context, using a more combinatorial approach.

Although, as we have already pointed out, good semigroups share  traits with the numerical semigroups, there are some important properties of the latter that cannot be generalized to them. For instance, they are not finitely generated as monoids, and they are not closed under finite intersections.
This makes the study of good semigroups much more difficult than the numerical ones.

Thus, a relevant part of the literature dedicated to these objects is concerned to find a suitable way to represent them by means of a finite set of data. 

For instance, for what concerns good semigroups which are also value semigroups, in \cite{semig:garcia, gruppe:waldi} singularities with only two branches are studied. In these papers, the finite set considered is the set of maximal elements (in \cite{multibranch:delgado}, it is possible to find a generalization of this approach to the case of more than two branches). 
In \cite{good:danna}, the authors considered a new approach that is still valid for  good semigroups not realizable as value semigroups of curves.
They firstly notice that the set of \emph{small elements} of the semigroup, that is, the finite set of elements between $0$ and the conductor of the semigroup with the usual partial order, completely describes it. Then they proved the uniqueness of the minimal subset $G \subsetneq\mathrm{Small}(S)$, called  \emph{minimal good generating system}, from which is possible to recover completely the semigroup $S$, if also the conductor of $S$ is known. 
Another interesting approach is the one presented in \cite{Carvalho:Semiring}, where the authors introduced the  semiring of values $\Gamma$ of an algebroid curve $R$ where also the values of the zero-divisors elements are considered ($v(0)=(\infty,\ldots,\infty$)). 
Thus $\Gamma$ contains the value semigroup of $R$ and $(\Gamma,+)$ is a semigroup setting $\gamma +
\underline{\infty}=\underline{\infty}$ for all $\gamma\in\Gamma$.
The key point is that $\Gamma$, equipped with the
tropical operations
$$ \bs{\alpha} \oplus \bs{\beta}=\text{min}\{\bs{\alpha},\bs{\beta}\}:=(\text{min}\{\alpha_1,\beta_1\},\ldots ,\text{min}\{\alpha_n,\beta_n\})\ \ \
\mbox{and} \ \ \ \bs{\alpha}\odot\bs{\beta}=\bs{\alpha}+\bs{\beta},$$ 
is a {\it finitely generated semiring}. This leads  the authors to introduce the concept of minimal standard basis.

The aim of this paper is to continue this kind of investigation, in order to find the smallest possible finite set that is able to encode some of the information of a good semigroup with two branches. Specifically, we introduce the concept of minimal set of representatives of a good subsemigroup $S$ of $\mathbb{N}^2$. Although a minimal set of representatives $\eta$ of $S$ does not univocally describe the semigroup (however $S$ is still among the minimal good semigroups containing $\eta$), it is possible to show that it stores relevant data. For instance, in the case of value semigroup, a system of representatives contains all the information regarding the value of a minimal system of generators of the corresponding ring. This leads us to generalize in a reasonable way, to the good semigroups of $\mathbb{N}^2$, the  concept of \textit{embedding dimension} that plays an important role in the numerical case.

The structure of the paper is the following.

In Section \ref{section1} we give all the basic definitions and we introduce all the  main tools of the paper. In particular, in Subsection \ref{section11} we recall the definition of good semigroup and we explain how to associate to a good semigroup $S$ of $\mathbb{N}^2$ 
a semiring $\Gamma_S$.
Then, in Proposition \ref{E2bis}, we prove that, in the case of value semigroups, our semiring coincides with the one given in \cite{Carvalho:Semiring}.
In Subsection \ref{section12} we define the concept of \textit{irreducible} and \textit{absolute element} of $\Gamma_S$, and in Theorem \ref{E7}, we prove that $\Gamma_S$ is generated as a semiring by its set $I_A$ of irreducible absolute elements generalizing to all good semigroups a result proved by Carvalho E. and Hernandes M.E. \cite[Thm 11, Cor 20]{Carvalho:Semiring} for the value semigroups of a ring. 

In Section \ref{section2}, we introduce the notation $S_{\eta}$ for the set of  the minimal good semigroups containing $\eta$. In Proposition \ref{boundcond} we give some conditions on $\eta$ in order to have finitely many elements in $S_{\eta}$.
Then, given a good semigroup $S$, a set $\eta$ is called a \textit{ system of representatives} of $S$ if $S \in S_{\eta}$.
This lets us to define the embedding dimension of a good semigroup $S$ as the smallest cardinality of a system of representatives of $S$.
Starting from this point we work on good semigroups of $\N^2$ in order to study the property of the embedding dimension.
In Subsection \ref{section21} we introduce the definition of \emph{track} of a good semigroup $S$ and with Lemma \ref{semcont} we show how to obtain a good semigroup $S'$ contained in $S$ by removing one of its tracks. Using this lemma we can compute an inferior bound for the embedding dimension. In Subsection \ref{section22} it is given the definition of reducibility of an element of $I_A(S)$ with respect to a subset $\eta \subseteq I_A(S)$.  Then,  Theorem \ref{E21} gives a way to use this concept in order to develop a strategy to find a superior bound for the embedding dimension.
In Subsection \ref{section23} we present a series of functions implemented in GAP \cite{GAP4} that, using the computational vantages of calculating the previous bounds, allow us to describe a fast algorithm to find the embedding dimension.
In the examples proposed in this section, for reasons of legibility and space, some verifications are not reported; these were made using functions written in GAP \cite{GAP4}.

Finally, Section \ref{section3} is dedicated to studying  whether the embedding dimension defined in $\mathbb{N}^2$ retains some of the features of the numerical case.
In particular in Theorem \ref{ring} we prove that a good semigroup $S$, realizable as a value semigroup, has embedding dimension greater or equal than the corresponding ring (as in the numerical case). Then we give some examples, when the previous inequality is strict, where it is possible to observe the limits of the combinatorial structure of a good semigroup that is not always able to store all the information contained in the ring in the same amount of data given by a system of generators.
In Subsection \ref{section32} we give the definition of levels of the Apéry set of a good semigroup as in \cite{DAGuMi}, and we use it to prove that $\operatorname{edim}(S)\leq e_1+e_2$, where $\bs{e}=(e_1,e_2)$ is the multiplicity vector of $S$ (extending the relation $\operatorname{edim}(S)\leq e$ of the numerical case and the corresponding relation for one-dimensional rings). This result also lets us to prove Corollary \ref{arf}, where we show that the Arf good semigroups of $\mathbb{N}^2$ have  maximal embedding dimension, generalizing another important property valid in the numerical case.

\section{Semiring associated to a good semigroup and Irreducible Absolutes}
\label{section1}
\subsection{Semiring $\Gamma_{S}$ and basic properties}
\label{section11}
We start this section recalling the definition of good semigroup introduced in \cite{anal:unr}.
\begin{defi}
\label{E1}
A submonoid $S$ of $(\N^n,+)$ is a \emph{good semigroup} if it satisfies the following properties:
\begin{itemize}
\item[(G1)] If $\bs{\alpha},\bs{\beta}\in S$, then $\min(\bs{\alpha};\bs{\beta}) = (\min\{\alpha_1,\beta_1\},\ldots, \min\{\alpha_n,\beta_n\})\in S$;
\item[(G2)] There exists $\bs{\delta} \in \N^n$ such that $S\supseteq \bs{\delta}+\N^n$;
\item[(G3)] If $(\bs{\alpha},\bs{\beta})\in S$; $\bs{\alpha}\neq \bs{\beta}$ and $\alpha_i=\beta_i$ for some $i\in\{1,\ldots,n\}$; then there exists $\bs{\epsilon} \in S$ such that $\epsilon_i>\alpha_i=\beta_i$ and $\epsilon_j\geq \min\{\alpha_j,\beta_j\}$ for each $j\neq i$ (and if $\alpha_j\neq \beta_j$, the equality holds).
\end{itemize}
\end{defi}
Furthermore, we always suppose to work with a \emph{local} good semigroup $S$, i.e. if $\bs{\alpha}=(\alpha_1,\ldots,\alpha_n)\in S$ and $\alpha_i=0$ for some $i\in\{1,\ldots,n\}$, then $\bs{\alpha}=\bs{0}$.
As a consequence of property (G2), the element $\bs{c}=\min\{\delta| S\supseteq \bs{\delta}+\N^n\}=(c_1,\ldots,c_n)$ is well defined and it is called \emph{conductor} of the good semigroup. We denote by $\leq$, the partial order on the elements of $S$ induced by the standard order on $\N^n$.
Furthermore, we denote by $\bs{e}=\min (S\backslash\{\bs{0}\})$ the \emph{multiplicity vector} of the good semigroup. 
In order to simplify the notation and some proofs, in this paper, we often work with good semigroups $S\subseteq \N^2$ but  most of the definitions and proofs remain true also in the general case.\\
According to the work of Carvalho and Hernandes \cite{Carvalho:Semiring}, we wish to introduce a semiring $\Gamma_S$ associated with the good semigroup $S\subseteq \N^2$.\\
We set $\overline{\N}=\N\cup\{\infty\}$, where $\infty$ is just a symbol that will correspond to the value of the element $0$ if the semigroup is the value semigroup of  a ring. 
We extend the natural order and the sum over $\N$ to $\overline{\N}$, setting respectively, $a<\infty$ for all $a\in \N$ and  $x+\infty=\infty+x=\infty$.\\

We set:
$$S_1^{\infty}=\{(a,\infty)\hspace{0.1cm}|\hspace{0.1cm}\exists \tilde{y}\in \N: (a,y)\in S\hspace{0.2cm} \forall y\geq \tilde{y}\};$$
$$S_2^{\infty}=\{(\infty,b)\hspace{0.1cm}|\hspace{0.1cm}\exists \tilde{x}\in \N: (x,b)\in S\hspace{0.2cm} \forall x\geq \tilde{x}\};$$
$$S^{\infty}=S_1^{\infty}\cup S_2^{\infty}\cup\{(\infty,\infty)\};$$
$$\Gamma_S=S\cup S^{\infty}.$$
If $\bs{\alpha}=(\alpha_1,\alpha_2),$ $\bs{\beta}=(\beta_1,\beta_2)\in \Gamma_S$, we set $\min\{\bs{\alpha},\bs{\beta}\}:=(\min\{\alpha_1,\beta_1\},\min\{\alpha_2,\beta_2\})$.\\
Now we define over $\Gamma_S$ the following tropical operations:
$$\oplus: \bs{\alpha} \oplus \bs{\beta} =\min\{\bs{\alpha},\bs{\beta}\}$$
$$\odot: \bs{\alpha}\odot \bs{\beta}= \bs{\alpha}+\bs{\beta} 
$$
It is easy to prove that, with these operations, $(\Gamma_S,\oplus,\odot)$ is a semiring.\\
We observe that, with the symbols $+$ and $\odot$, we denoted exactly the same operation on $\Gamma_S$. For this reason these two symbols will be used with the same meaning in the following. \\

Now we recall some facts and fix some notations that will be useful for the following.\\ 
Let be $R=\K[\![x_1,\ldots,x_n]\!]/Q$ a two-branches algebroid curve, where $Q=P_1\cap P_2$ is an ideal of $\K[\![x_1,\ldots,x_n]\!]$ such that $P_1$,$P_2$ are prime ideals.\\
We can embed $R\hookrightarrow R_1\times R_2$ where $R_i=\K[\![x_1,\ldots,x_n]\!]/P_i$, $i=1,2$. Furthermore $R\hookrightarrow\overline{R}\cong \overline{R_1}\times \overline{R_2}\cong \K[\![t_1]\!]\times \K[\![t_2]\!]$.
Given $r\in R$, $r=(r_1,r_2)\in \K[\![t_1]\!]\times \K[\![t_2]\!]$ that is a product of DVRs, so we can associate to each element of $R$ a valuation. 
If $v_i$ is the valuation function on $\K[\![t_i]\!]$, we set:
$$v_i(r)=\begin{cases}
v_i(r_i) &\text{ if } r_i\neq 0\\
\infty &\text{ if } r_i=0
\end{cases}
$$
and $v(r)=(v_1(r),v_2(r))$.\\
According to the notation of Carvalho and Hernandes \cite{Carvalho:Semiring}, we introduce the following sets:
$$\Gamma_{S_i}=\{v_i(r)\hspace{0.1cm}|\hspace{0.1cm}r\in R\}\subseteq \overline{\N};$$
$$S_i=\{v_i(r)|r\text{ is not a zerodivisor in } R\}\subseteq \N;$$
$$\Gamma_R=\{\bs{v(r)}:=(v_1(r),v_2(r)), r\in R\}\subseteq \overline{\N}^2;$$
$$S=\{\bs{v(r)}:=(v_1(r),v_2(r))\hspace{0.1cm}|\hspace{0.1cm} r \text{ is not a zerodivisor in } R\}\subseteq \N^2.$$
$\Gamma_R$ and $S$ will be called respectively \emph{semiring of values} and \emph{semigroup of values} associated to $R$.
It is  easy to observe that $S=\Gamma_R\cap \N^2$.\\ \\
At this point, we wish to prove that, if $R$ is a two-branches algebroid curve, and $S$ is its semigroup of values, then $\Gamma_S=\Gamma_R$.\\ 
\begin{lem}
\label{E2} 
The following statements hold:
\begin{itemize}
\item[i)] $(a,\infty)\in \Gamma_S$ if and only if $(a,y)\in S$ for any $y\geq c_2$.
\item[ii)] $(\infty,b)\in \Gamma_S$ if and only if $(x,b)\in S$ for any $x\geq c_1$.
\end{itemize}
\end{lem}
\begin{proof}
We prove $i)$, the other statement is analogue. If $(a,\infty)\in \Gamma_S$, then there exists $\tilde{y}\in \N$ such that $(a,\tilde{y}), \ldots,$ \newline $ (a,\tilde{y}+ \nolinebreak n)\in S$ for any $n\in \N$.
If $\tilde{y}\leq c_2$ the statement is proved, otherwise $\tilde{y}=c_2+n$, with $n\in \N$. Since $S$ is a good semigroup, for all $i<n$, $a<c_1$, we have that $(a,c_2+i)=\min\{(a,\tilde{y}),(c_1,c_2+i)\}\in S$.
\end{proof}
\begin{prop}
\label{E2bis}
If $R$ is a two-branches algebroid curve and $S$ is its semigroup of values, then $\Gamma_S=\Gamma_R$.
\end{prop}
\begin{proof}
We have observed that $S=\Gamma_R\cap\N^2$, thus we need to prove that $\Gamma_R\backslash S=S^{\infty}$. If $\bs{\alpha}\in \Gamma_R\backslash S$, we can write $\bs{\alpha}=v(r)$, where $r$ is a zerodivisor in $R$ or $r=0$; in both cases we have $r\in P_1\cup P_2$. If $r=0$, $v(r)=(\infty,\infty)$; if $r\in P_1$, then $r=0$ in $R_1$, $v_1(r)=\infty$, hence $\bs{\alpha}\in S_2^{\infty}$; if $r\in P_2$, then $r=0$ in $R_2$, $v_2(r)=\infty$, hence $\bs{\alpha}\in S_1^{\infty}$. If $\bs{\alpha}\in S^{\infty}$, without loss of generality, we can suppose $\bs{\alpha}\in S_2^{\infty}$, we can write $\bs{\alpha}=(\infty,b)$, and, as a consequence of Lemma \ref{E1}, $(c_1,b)\in S$. Since $S=v(R)$ and the conductor ideal is $\mathcal{C}=(t^{c_1},u^{c_2})(\K[\![t]\!]\times\K[\![u]\!])$, there exists an element in $R$ of the form $(t^{c_1},b_y(u))$ with $v(b_y(u))=b$.
Since the element $(t^{c_1},0)\in R$, we have that the element $(0,b_y(u))\in R$, thus $(\infty,b)\in \Gamma_R$.
\end{proof}
\subsection{A system of generators of $\Gamma_S$ as a semiring}
\label{section12}
\begin{defi}
\label{E3}
We will say that an element $\bs{\alpha}\in \Gamma_S\setminus\{\bs{0}\}$ is irreducible if, from $\bs{\alpha}=\bs{\beta}+\bs{\gamma}$, with $\bs{\beta},\bs{\gamma}\in \Gamma_S$, it follows $\bs{\alpha}=\bs{\beta}$ or $\bs{\alpha}=\bs{\gamma}$. An element that is not irreducible will be said reducible.
\end{defi}
We denote by $I(S)$ the set of irreducible elements of $\Gamma_S$.
\begin{oss}
\label{E4}
We observe that:
\begin{enumerate}
\item If $\bs{\alpha}=(a,b)\in \Gamma_S$ with $a\geq c_1+e_1$ and $b\geq c_2+e_2$, then $\bs{\alpha}$ is reducible.
\item If $\bs{\alpha}=(a,\infty)\in \Gamma_S$ with $a\geq c_1+e_1$, then $\bs{\alpha}$ is reducible.
\item If $\bs{\alpha}=(\infty, b)\in \Gamma_S$ with $b\geq c_2+e_2$, then $\bs{\alpha}$ is reducible.  
\end{enumerate}
\end{oss}
Given a good semigroup $S\subseteq \N^2$, and an element $\bs{\alpha}\in \N^2$, following the notation in \cite{anal:unr}, we set:
\begin{eqnarray*}
	\Delta_i(\bs{\alpha})&:=&\{\bs{\beta}\in \Z^{2}|\alpha_i=\beta_i \text{ and } \alpha_j<\beta_j \text{ for } j\neq i\}\\
	\Delta(\bs{\alpha})&:=&\Delta_1(\bs{\alpha})\cup\Delta_2(\bs{\alpha})\\
	\Delta_i^S(\bs{\alpha})&:=&S\cap \Delta_i(\bs{\alpha})\\
	\Delta^S(\bs{\alpha})&:=&S\cap\Delta(\bs{\alpha}).
\end{eqnarray*}
Furthermore we define:
\begin{eqnarray*} {}_i\Delta(\bs{\alpha})&:=&\{\bs{\beta}\in \Z^{2}|\alpha_i=\beta_i \text{ and } \beta_j<\alpha_j \text{ for } j\neq i\}\\
	_i\Delta^S(\bs{\alpha})&:=&S\cap _i\Delta(\bs{\alpha}).\end{eqnarray*}
	Extending the previous definitions to infinite elements of $\overline{\N}^2$, we set
	\begin{eqnarray*}
	_1\Delta((\alpha_1,\infty))&:=&\{\bs{\beta}\in \Z^{2}|\beta_1=\alpha_1 \}\\
		_2\Delta((\alpha_1,\infty))&:=&\emptyset\\
	_1\Delta((\infty,\alpha_2))&:=&\emptyset\\
		_2\Delta((\infty,\alpha_2))&:=&\{\bs{\beta}\in \Z^{2}|\beta_2=\alpha_2 \}
		\\
	_i\Delta^S(\bs{\alpha})&:=&S\cap {}_i\Delta(\bs{\alpha}).
\end{eqnarray*}
\begin{defi}
\label{E5}
An element $\bs{\alpha}\in \Gamma_S$ will be said absolute in $\Gamma_S$ if $\bs{\alpha}\in S$ and $\Delta^S(\bs{\alpha})=\emptyset$ (finite absolute), or if $\bs{\alpha}\in S^{\infty}$ (infinite absolute).
\end{defi}
\begin{oss}
\label{assirroplus}
We observe that an element $\bs{\alpha}\in \Gamma_S$ is an absolute in $\Gamma_S$ if and only if it is irreducible with respect to the operation $\oplus$. If we suppose that $\bs{\alpha}\in \Gamma_S$ is not an absolute, then $\Delta^S(\bs{\alpha})\neq\emptyset$, hence there exists $\bs{\beta}\in \Delta_i^S(\bs{\alpha})$, with $i\in\{1,2\}$. Therefore, by property (G3) of the good semigroups, there exists $\bs{\gamma}\in \Delta_{3-i}^S(\bs{\alpha})$, hence we would have $\bs{\alpha}=\bs{\beta}\oplus\bs{\gamma}$, which is a contradiction. If we suppose that an element $\bs{\alpha}\in \Gamma_S$ is such that $\bs{\alpha}=\bs{\beta}\oplus\bs{\gamma}$ with $\bs{\beta},\bs{\gamma}\neq \bs{\alpha}$, then $\bs{\alpha}\in S$ and $\Delta^S(\bs{\alpha})\neq \emptyset$.
\end{oss}
We denote by $A_f(\Gamma_S)$ the set of \emph{finite absolutes} in $\Gamma_S$, by $A^{\infty}(\Gamma_S)$ the set of \emph{infinite absolutes} in $\Gamma_S$ and by $A(\Gamma_S)$ the set of all \emph{absolutes} in $\Gamma_S$.
We call $I_{A_f}(\Gamma_S)$ the set of \emph{finite irreducible absolutes} in $\Gamma_S$, $I_{A}^{\infty}(\Gamma_S)$ the set of \emph{infinite irreducible absolutes} in $\Gamma_S$ and $I_A(\Gamma_S)$ the set of all \emph{irreducible absolutes} in $\Gamma_S$.
\begin{oss}
By Remark \ref{assirroplus},  $I_A(S)$ can be seen as the set
of the elements of $\Gamma_S$ that are irreducible with respect to both the operations defined in it. Notice that this interpretation lets us to naturally generalize the concept of irreducible absolute elements to good subsemigroups of $\mathbb{N}^n$, with $n>2$.
\end{oss}
As a consequence of the Remark \ref{E4}, the set of irreducible absolutes is finite.
Now we introduce other sets that will be considered in the following:
\begin{eqnarray*}
\operatorname{small}(S)=&\{(a,b)\in S|a\leq c_1, b\leq c_2\};\\
\operatorname{small}(\Gamma_S)=&\operatorname{small}(S)\cup\{(\infty,b)\in S_2^{\infty},b\leq c_2\}\cup\{(a,\infty)\in S_1^{\infty},a\leq c_1\};\\
B_1^{\infty}(\Gamma_S)=&\{(a,\infty)\in \Gamma_S|c_1< a\leq c_1+e_1\}\subseteq S_1^{\infty};\\
B_2^{\infty}(\Gamma_S)=&\{(\infty,b)\in \Gamma_S|c_2< b\leq c_2+e_2\}\subseteq S_2^{\infty};\\
B^{\infty}(\Gamma_S)=&B_1^{\infty}(\Gamma_S)\cup B_2^{\infty}(\Gamma_S)\subseteq S^{\infty}(C).
\end{eqnarray*}

The sets $\operatorname{small}(S)$, $\operatorname{small}(\Gamma_S)$, $B^{\infty}(\Gamma_S)$ will be said respectively: \emph{small elements} of $S$, \emph{small elements} of $\Gamma_S$ and \emph{beyond elements} of $\Gamma_S$.

	\begin{figure}[h]
		\begin{center}
	\tikzset{mark size=1}
	\pgfplotsset{ticks=none}
	\begin{tikzpicture}[scale=1.0]
	\begin{axis}[xmin=0, ymin=0, xmax=20, ymax=20]
	\addplot[very thick]coordinates{(11,0) (11,20)}; \addplot[very thick]coordinates{(0,11) (20,11)};
	\addplot[very thick]coordinates{(11,16) (20,16)};
	\addplot[very thick]coordinates{(16,11) (16,20)};
	\addplot [pattern = north east lines, draw=white]coordinates{(16,16)(20,16)(20,20)(16,20)(16,16)};
	\addplot [pattern = north east lines, draw=white]coordinates{(11,16)(16,16)(16,20)(11,20)(11,16)};
	\addplot [pattern = north east lines, draw=white]coordinates{(16,11)(20,11)(20,16)(16,16)(16,11)};
	\addplot[only marks] coordinates{(11,11) (16,16)}; 
	\node[text width=0cm, font=\tiny] at (14,15.5) 
	{$\bs{c+e}$};
	\node[text width=0cm, font=\tiny] at (10.5,10.5) 
	{$\bs{c}$};
	\node[text width=3cm, font=\footnotesize] at (8,5.5) 
	{$small(S)$};
	\node[text width=3cm, font=\footnotesize] at (16.2,5.5) 
	{$S_1^{\infty}\cap small(\Gamma_S)$};
	\node[text width=3cm, font=\footnotesize] at (6,15.5) 
	{$S_2^{\infty}\cap small(\Gamma_S)$};
	\node[text width=3cm, font=\footnotesize] at (17,13.5) 
	{$B(\Gamma_S)$};
	\end{axis}
	\end{tikzpicture}
	\caption{A graphic representation of $\Gamma_S$'s elements}
\end{center}
\end{figure}
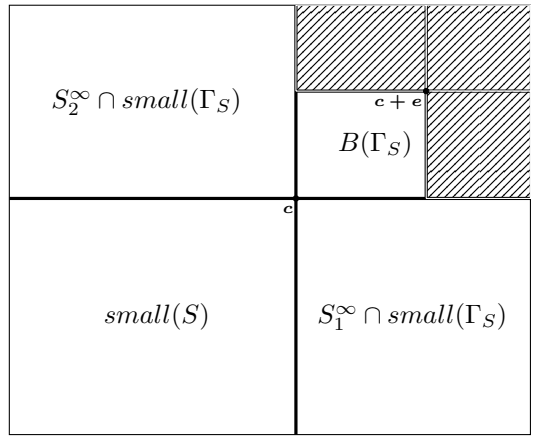
\label{Figure 1}

\begin{oss}
\label{E6} It is easy to observe the following facts:

\begin{enumerate}
\item[i)] Each element in the semiring can be written as a tropical product of irreducible elements, i.e. if $\bs{\alpha}\in \Gamma_S$, $\bs{\alpha}=\bs{\alpha_1}\odot\ldots\odot\bs{\alpha_n}$ where $\bs{\alpha_i}\in I(\Gamma_S)$.
\item[ii)] Each element in the semiring can be written as a tropical sum of two absolute elements, i.e. if $\bs{\beta}\in \Gamma_S$, $\bs{\beta}=\bs{\beta_1}\oplus\bs{\beta_2}$ where $\bs{\beta_1},\bs{\beta_2}\in A(\Gamma_S)$.
\end{enumerate}
\end{oss}

Now we prove that the set of irreducible absolutes generates $\Gamma_S$ as  a semiring.

\begin{teo}
\label{E7}
$(\Gamma_S,\oplus,\odot)$ is generated as a semiring by the irreducible absolutes, i.e. if $\bs{\alpha}\in \Gamma_S\setminus\{\bs{0}\}$, $$\bs{\alpha}= \bigoplus_{i=1}^m(\bigodot_{j=1}^n\bs{\gamma_{j_i}}),$$ with $\bs{\gamma_{j_i}}\in I_A(S)$.
\end{teo}
\begin{proof}
First of all, we observe that we can reduce to prove the thesis only for the elements $\bs{\alpha}\in \operatorname{small}(\Gamma_S)\cup B(\Gamma_S)$. Indeed, if $\bs{\alpha} \notin \operatorname{small}(\Gamma_S)\cup B(\Gamma_S)$, then there exists $k\in \N$ such that $\bs{\beta}=\bs{\alpha}-k\bs{e}\in \operatorname{small}(\Gamma_S)\cup B(\Gamma_S)$. In this case we would have $\bs{\alpha}=\bs{\beta}\odot k\bs{e}$, where $\bs{\beta} \in \operatorname{small}(\Gamma_S)\cup B_S$ and $\bs{e}$ is trivially irreducible.\\
We can reduce again the proof only for the elements $\bs{\alpha}\in I(\Gamma_S)\cap S$ (finite irreducibles). In fact, if $\bs{\alpha}$ is reducible, by Remark \ref{E6}, we can write $\bs{\alpha}=\bs{\alpha^{(1)}}\odot\ldots\odot\bs{\alpha^{(n)}}$, with $\bs{\alpha^{(i)}}$ irreducibles. Furthermore, we observe that if $\bs{\alpha^{(i)}}\in S^{\infty}$, then $\bs{\alpha^{(i)}}\in I_{A}(\Gamma_S)$; thus we can write:
$$\bs{\alpha}=\bs{\alpha^{(1)}}\odot\ldots\odot\bs{\alpha^{(f)}}\odot\Big(\displaystyle{\odot_{\bs{\gamma}\in I_A(S)}\bs{\gamma}}\Big)$$
where $\bs{\alpha^{(i)}}\in I(\Gamma_S)\cap S$. If we prove the thesis for the elements $\bs{\alpha^{(i)}}$ with $i\in\{1,\ldots,f\}$, using the distributive property of $\odot$ with respect to $\oplus$, the result is true also for $\bs{\alpha}$.
Therefore we can suppose $\bs{\alpha}\in I(\Gamma_S)\cap S$ and prove the thesis. By Remark \ref{E6}, we can write $\bs{\alpha}=\bs{\beta}\oplus\bs{\gamma}$ with $\bs{\beta}=(\beta_1,\beta_2)\in A$, $\bs{\gamma}=(\gamma_1,\gamma_2)\in A$ and we can assume $\beta_1=\alpha_1\leq \gamma_1$ and $\gamma_2=\alpha_2\leq \beta_2$.\\
We consider
\begin{eqnarray*}
\bs{\beta}&=&\bs{\beta^{(1)}}\odot\ldots\odot\bs{\beta^{(n)}},\\
\bs{\gamma}&=&\bs{\gamma^{(1)}}\odot\ldots\odot\bs{\gamma^{(m)}},
\end{eqnarray*}
the decompositions in irreducible elements of $\bs{\beta}$ and $\bs{\gamma}$.
We define $\bs{\beta'^{(i)}}=\bs{\beta^{(i)}}\oplus \bs{\gamma}$, for all $i\in \{1,\ldots,n\}$ and $\bs{\gamma'^{(j)}}=\bs{\gamma^{(j)}}\oplus \bs{\beta}$ for all $j\in \{1,\ldots,m\}$. Defining $\bs{\beta'}=\bs{\beta'^{(1)}}\odot\ldots\odot\bs{\beta'^{(n)}}$, $\bs{\gamma'}=\bs{\gamma'^{(1)}}\odot\ldots\odot\bs{\gamma'^{(m)}}$, it is easy to observe that $\beta'_1=\beta_1$ and $\gamma'_2=\gamma_2$, thus we have $\bs{\alpha}=\bs{\beta'}\oplus\bs{\gamma'}$.\\
We can definitely write: $$\bs{\alpha}=(\bs{\beta'^{(1)}}\odot\ldots\odot\bs{\beta'^{(n)}})\oplus(\bs{\gamma'^{(1)}}\odot\ldots\odot\bs{\gamma'^{(m)}}),$$
where each $\bs{\beta'^{(i)}}$ and $\bs{\gamma'^{(j)}}$ is strictly smaller than $\bs{\alpha}$ (that is $\bs{\gamma'^{(j)}} \leq \bs{\alpha}$ and $\bs{\gamma'^{(j)}}  \neq \bs{\alpha}$). If we express each of these elements as a tropical product of irreducibles, we can write $\bs{\alpha}$ as a tropical sum of tropical products, where all the terms are irreducible and strictly smaller than $\bs{\alpha}$. This means that if we repeat the same argument on each element in this expression, in a finite number of iteration we will surely obtain the required expression.
\end{proof}

\begin{oss}
In the case of good semigroups that are value semigroup of a ring, the theorem above follows by \cite[Thm 11]{Carvalho:Semiring} and \cite[Thm 19]{Carvalho:Semiring}.\\
But we recall that not all good semigroups are value semigroup of a ring (for an example cf.\cite[Example 2.16]{anal:unr}).\\
Thus, the previous theorem generalizes this property to all semirings obtained by semigroups of $\N^2$, also if they are not value semigroup of a ring.
\end{oss}

\section{Embedding dimension of a good semigroup}
\label{section2}
It is a well known fact that every numerical semigroup $S\subseteq \N$ admits a unique minimal system of generators as a monoid and the embedding dimension of the numerical semigroup is defined as the number of these generators. This name follows from the fact that it is equal to the embedding dimension of the monomial curve associated with the numerical semigroup.\\
Now we will define a set of vectors that, although it does not uniquely determine a good semigroup, will allow us to give a definition of embedding dimension of a good semigroup. This embedding dimension, in the case of good semigroup of $\N^2$, will satisfy some of the properties that are valid in the case of numerical semigroups.

Starting from this point, in order to lighten the notations,  when we  consider a good semigroup $S$, we  suppose that it coincides with the semiring $\Gamma_S$, i.e. we treat the infinite elements as elements of $S$.
Given a set of vectors $\eta\subseteq \overline{\N}^n$, we denote by $\sgen{\eta}$ the semiring generated by $\eta$.
Furthermore, given a set of vectors $\eta\subseteq \overline{\N}^n$, we denote by $S_{\eta}$ the family of all the good semigroups  containing $\eta$ and that are minimal with respect to the set inclusion. $S_{\eta}$ can be finite, infinite or empty as in the following example.

\begin{ex}
\label{E9}
Let us consider $\eta=\{[2,2],[3,3]\}\subseteq \N^2$, and suppose that there exists a good semigroup $S\in S_{\eta}$.\\
First of all we prove that, for any $n\in \N\backslash\{1\}$, we have $(n,n)\in S$. In fact, it is easy to observe that each natural number $n\neq 1$ can be written as $n=2\alpha+3\beta$, with $\alpha,\beta\in \N$. Hence we can write $(n,n)=(2\alpha+3\beta,2\alpha+3\beta)=\alpha(2,2)+\beta(3,3)\in S$.\\
We denote by $\bs{c(S)}=(c_1,c_2)$ the conductor of $S$. If $c_1=1$, we have that $(1,2)=\min\{(1,c_2),$ $(2,2)\}\in S$; hence, as a consequence of properties (G1) and (G3) of the good semigroups, either $\bs{c(S)}=(1,2)$ or $S=\N^2$. In both cases, if we consider $S'$ such that $\operatorname{small}(S')=\{(0,0),(2,2)\}$ we have that $S'$ is a good semigroup containing $\eta$ and such that $S'\subset S$; but this contradicts the minimality of $S$. Therefore we have obtained $c_1\neq 1$ and, using the same argument, we can suppose $c_2\neq 1$.\\
If $c_1>1$ and $c_2>1$ we prove that $\bs{c(S)}=(c,c)$, with $c\in \N$. Let us assume by contradiction that $\bs{c(S)}=(c_1,c_2)$ with $c_1<c_2$; in this case, there exists $\bs{\alpha}=(\alpha_1,\alpha_2)$ with $\alpha_1\geq c_1$, $c_1\leq \alpha_2<c_2$ such that $\bs{\alpha}\notin S$. If $\alpha_1\leq \alpha_2$, we would have $\bs{\alpha}=\min\{(\alpha_1,c_2),(\alpha_2,\alpha_2)\}\in S$, hence  we necessarily have $\alpha_1>\alpha_2$. Now we observe that $(c_1,\alpha_2)=\min\{\bs{c(S)},(\alpha_2,\alpha_2)\}\in S$ and by property (G3) of good semigroups applied to $\bs{c(S)}$ and $(c_1,\alpha_2)$, there exists $(x_1,\alpha_2)\in S$ with $x_1>c_1$. If $x_1\geq \alpha_1$, we would have $\bs{\alpha}=\min\{(x_1,\alpha_2),(\alpha_1,c_2)\}\in S$ that is a contradiction. Thus we necessarily have $x_1<\alpha_1$. Now, if we consider $(x_1,\alpha_2), (x_1,c_2)\in S$, using again property (G3), we observe that there exists $(x_2,\alpha_2)\in S$ with $x_2>x_1$. We can repeat this argument until we  find an element $(x_i,\alpha_2)\in S$ with $x_i\geq \alpha_1$. In this case we obtain $\bs{\alpha}=\min\{(x_i,\alpha_2),(\alpha_1,c_2)\}\in S$, that is a contradiction.\\
Now, by repeatedly using the properties (G2) and (G3), it is easy to observe that, $\operatorname{small}(S)=\{(0,0),(2,2),(3,3),\ldots, (c-1,c-1),(c,c)\}$. If we define $S'$ such that $\operatorname{small}(S')=\{(0,0),(2,2),$ $(3,3),\ldots,  (c,c),(c+1,c+1)\}$, we have found a minimal good semigroup containing $(2,2),(3,3)$ and strictly contained in $S$, in contradiction with the minimality of $S$.
\end{ex}
The following proposition gives a condition that guarantees that $S_{\eta}$ is finite.

\begin{prop} \label{boundcond}
	Suppose  we have $ \eta=\{ \bs{\eta}^{(1)}=(\eta_1^1,\ldots,\eta_n^1),\ldots,\bs{\eta}^{(k)}=(\eta_1^k,\ldots,\eta_n^k) \}\subseteq \mathbb{N}^n$. 
	
	Then the set $S_{\eta}$ is finite if  the following conditions hold:
	\begin{itemize}
		\item $\gcd\left\{ \eta_i^h, h=1,\ldots,k \right\}=1$ for $i=1,\ldots,n;$
		\item For all $i,j \in \left\{1,\ldots,n \right\}$ with $i \neq j$  there exists a $l \in \{ 1,\ldots,k\} $ such that $\eta_i^l \neq \eta_j^l$. 
	\end{itemize}
\end{prop}
\begin{proof}
We denote by $ \sgen{\eta}$ the semiring generated by  $\eta$.
We claim that for each $i=1,\ldots,n$,  we can obtain two vectors $\bs{\alpha}^{(i)}=(\alpha_1^i,\ldots,\alpha_n^i)$ and $\bs{\beta}^{(i)}=(\beta_1^i,\ldots,\beta_n^i)$ in $\sgen{\eta}$  such that $$ \alpha_i^i=\beta_i^i \textrm{ and }\alpha_i^j<\beta_i^j \textrm { for all } j \neq i. $$

We will prove this fact by induction on $n$.
\begin{itemize}
	\item \textbf{Base case} $n=2$.
	Suppose that $i=1$. By the second property assumed on the set $\eta$, there exists a $ \bs{\eta}^{(l)} \in \eta$ such that $ \eta_1^l \neq \eta_2^l$. Then $\eta$ must contain a vector $\bs{\eta}^{(m)} $ such that $ \frac{\eta_2^m}{\eta_1^m} \neq  \frac{\eta_2^l}{\eta_1^l} $.
	We assume by contradiction that  $\frac{\eta_2^h}{\eta_1^h} =  \frac{\eta_2^l}{\eta_1^l} \neq 1$ for all $h=1,\ldots,k$. If $\eta_1^l$ did not divide $\eta_2^l$,  it would follow  from $\eta_2^h=\frac{\eta_2^l}{\eta_1^l}\eta_1^h$ that $\eta_1^l$ divides $\eta_1^h$ for all $h=1,\ldots,k$. Hence $\eta_1^l$ would divide $\gcd\left\{\eta_1^h, h=1,\ldots,k \right\}=1$; but this contradicts the first assumption on the set $\eta$. Therefore we have 
    $\frac{\eta_2^l}{\eta_1^l}\in \N$. Since the integer $\frac{\eta_2^l}{\eta_1^l}$, divides $\eta_2^h$ for all $h=1,\ldots,k$; it divides $\gcd\left\{\eta_2^h, h=1,\ldots,k \right\}=1$ but this is a contradiction.
    	
	Then, we  consider $\bs{\eta}^{(m)} $ such that $ \frac{\eta_2^m}{\eta_1^m} \neq  \frac{\eta_2^l}{\eta_1^l} $ and the vectors
	$$ \bs{\alpha}^{(1)}=(\eta_1^l\eta_1^m,\eta^l_2 \eta_1^m),  \qquad \bs{\beta}^{(1)}=(\eta_1^l \eta_1^m,\eta_1^l\eta_2^m) $$
	satisfy our condition because $\eta^l_2 \eta_1^m \neq \eta_1^l\eta_2^m$ and they belong to $\sgen{\eta}$ . 
	For $i=2$ we can use the same strategy.
	\item \textbf{ Inductive step}:
	Let us suppose that the claim is true for $n-1$ and we prove it for $n$. We suppose that $i=1$ (the other cases can be treated in the same way).
	We consider the set $\tilde{\eta}=\left\{ \bs{\eta}^{(h)}=(\eta_1^h,\ldots,\eta_{n-1}^h), h=1,\ldots,k\right\}$ that satisfies the conditions of the theorem. Then, by the inductive step, it easily follows that in $\sgen{\eta}$ there exist two vectors $\bs{\gamma}^{(1)}=(\gamma_1^1,\ldots, \gamma_n^1)$ and  $\bs{\delta}^{(1)}=(\delta_1^1,\ldots, \delta_n^1)$ such that  $$ \gamma^1_1=\delta^1_1 \textrm{ and }\gamma^1_j<\delta^1_j \textrm { for all } j=2,\ldots,n-1. $$
	If $\gamma^1_n<\delta^1_n$, then the claim is true for $\bs{\alpha}^{(1)}=\bs{\gamma}^{(1)}$ and  $\bs{\beta}^{(1)}=\bs{\delta}^{(1)}$. If  $\gamma^1_n>\delta^1_n$, we consider $\bs{\alpha}^{(1)}=\min(2\bs{\gamma}^{(1)},2\bs{\delta}^{(1)})$ and $\bs{\beta}^{(1)}=\bs{\gamma}^{(1)}\odot\bs{\delta}^{(1)}$.
	In fact we have $\alpha^1_1=2\gamma^1_1=\beta^1_1$. If $j \in \left\{ 2,\ldots,n-1\right\}$, then $\alpha^1_j=2\gamma^1_j<\gamma^1_j+\delta^1_j=\beta^1_j$. Finally, we have $\alpha^1_n=2\delta^1_n<\gamma^1_n+\delta^1_n=\beta^1_n$. Thus suppose that  $\gamma^1_n=\delta^1_n$.
	In this case we can consider  $\overline{\eta}=\left\{ \bs{\eta}^{(h)}=(\eta_1^h, \eta_3^h,\ldots,\eta_{n}^h), h=1,\ldots,k\right\}$. By the inductive step there exist two vectors
	$\bs{\gamma}^{(2)}=(\gamma_1^2,\ldots, \gamma_n^2)$ and  $\bs{\delta}^{(2)}=(\delta_1^2,\ldots, \delta_n^2)$ $  \in \sgen{\eta}$  such that  $$ \gamma^2_1=\delta^2_1 \textrm{ and }\gamma^2_j<\delta^2_j \textrm { for all } j=3,\ldots,n. $$
	Then, as we have just seen, if $\gamma^2_2\neq \delta^2_2 $  the claim is true. Therefore we suppose that  $\gamma^2_2= \delta^2_2 $. Then, it is very easy to check that the claim is true with $\bs{\alpha}^{(1)}=\bs{\gamma}^{(1)}\odot\bs{\gamma}^{(2)}$ and  $\bs{\beta}^{(1)}=\bs{\delta}^{(1)}\odot\bs{\delta}^{(2)}$.
\end{itemize}
Now, we denote by $c^{(i)}$ the conductor of the numerical semigroup generated by $\left\{ \eta_i^h: h=1,\ldots,k \right\}$ and we choose $\bs{\alpha}^{(i)}=(\alpha_1^i,\ldots,\alpha_n^i)$ and $\bs{\beta}^{(i)}=(\beta_1^i,\ldots,\beta_n^i)$ in $\sgen{\eta}$ as in the previous claim. We will prove that for each $i \in \left\{1,\ldots,n \right\}$ there exist $c_{i,j}$ for $j=1,\ldots,i-1,i+1,\ldots,n$ such that the vectors 
$$ \bs{\sigma}^i(y)=(c_{i,1},\ldots,c_{i,i-1},c^{(i)}+\alpha^i_i+y,c_{i,i+1},\ldots,c_{i,n}) \in S, $$
for each $S \in S_\eta$, and $y \in \mathbb{N}$.
If this is true then it is clear that  $$ c_{\eta}=\bigodot_{i=1}^n{\bs{\sigma}^i(0)}+\mathbb{N}^n \subseteq S,$$
for all $S \in S_{\eta}$.

Suppose that $i=1$ (the proof is identical in the other cases). Let us consider an arbitrary $S \in S_{\eta}$.  We obviously have $\sgen{\eta} \subseteq S$. We will denote by  $m=\alpha^1_1=\beta^1_1$.  Since $c^{(1)}$ is the conductor of $\langle \left\{ \eta_1^h: h=1,\ldots,k \right\}  \rangle$, we can find  the vectors:
$$ \bs{\sigma}^{(h)}=(\sigma_1^h,\ldots,\sigma_n^h)\in \sgen{\eta}, \qquad \textrm{ for } h=0,\ldots,m-1, $$
such that $\sigma_1^h=c^{(1)}+h$ for all $h=0,\ldots,m-1. $

For each $i=0,\ldots,m-1$ we consider $\bs{\lambda}^{(i)}=\bigoplus_{k=i}^{m-1}{\bs{\sigma}^{(k)}}$. Then we have $\bs{\lambda}^{(0)} \leq \ldots \leq \bs{\lambda}^{(m-1)}$ and, if $\bs{\lambda}^{(h)}=(\lambda_1^h,\ldots,\lambda_n^h)$, then $\lambda_1^h=c^{(1)}+h$.

Now we want to show that $ (c^{(1)}+m+y,\lambda_2^0+\alpha_2^1,\ldots,\lambda_n^0+\alpha_n^1) \in S$ for each $y \in \mathbb{N}$.

We notice that \begin{eqnarray*} \bs{\lambda}^{(0)}\odot\bs{\alpha}^{(1)}&=(c^{(1)}+m,\lambda_2^0+\alpha_2^1,\ldots,\lambda_n^0+\alpha_n^1) \in S, \\
	\bs{\lambda}^{(0)}\odot\bs{\beta}^{(1)}&=(c^{(1)}+m,\lambda_2^0+\beta_2^1,\ldots,\lambda_n^0+\beta_n^1) \in S,
\end{eqnarray*}
thus, recalling that $\alpha_j^1<\beta_j^1$ for all $j=2,\ldots,n$, it follows by (G3) that there exists $x>c^{(1)}+m$ such that $(x,\lambda_2^0+\alpha_2^1,\ldots,\lambda_n^0+\alpha_n^1) \in S.$

Now we consider $$\bs{\lambda}^{(1)}\odot \bs{\beta}^{(1)}=(c^{(1)}+m+1,\lambda_2^1+\beta_2^1,\ldots,\lambda_n^1+\beta_n^1) \in S.$$

Since $\lambda_h^0 \leq \lambda_h^1$ for all $h=2,\ldots,n$ and $\alpha_j^1<\beta_j^1$ for all $j=2,\ldots,n$, we have
$$(x,\lambda_2^0+\alpha_2^1,\ldots,\lambda_n^0+\alpha_n^1) \oplus (\bs{\lambda}^{(1)}\odot\bs{\beta}^{(1)})=(c^{(1)}+m+1,\lambda_2^0+\alpha_2^1,\ldots,\lambda_n^0+\alpha_n^1) \in S. $$

Now, as before, from 
\begin{eqnarray*} \bs(c^{(1)}+m+1,\lambda_2^0+\alpha_2^1,\ldots,\lambda_n^0+\alpha_n^1) \in S,\\
	(c^{(1)}+m+1,\lambda_2^0+\beta_2^1,\ldots,\lambda_n^0+\beta_n^1) \in S,
\end{eqnarray*}
we can deduce that  there exists $x>c^{(1)}+m+1$ such that $(x,\lambda_2^0+\alpha_2^1,\ldots,\lambda_n^0+\alpha_n^1) \in S.$

Repeating the previous considerations and using the fact that $\bs{\lambda}^{(0)} \leq \bs{\lambda}^{(i)}$  for each $i\leq m-1$, we can show that 
$$  (c^{(1)}+m+y,\lambda_2^0+\alpha_2^1,\ldots,\lambda_n^0+\alpha_n^1) \in S,$$ 
for all $y=0,\ldots,m-1$.
Now, we can consider \begin{eqnarray*}	\bs{\lambda}^{(0)}\odot\bs{\beta}^{(1)}&=&(c^{(1)}+m,\lambda_2^0+\beta_2^1,\ldots,\lambda_n^0+\beta_n^1) \in S \\ \ldots \\ \bs{\lambda}^{(m-1)}\odot\bs{\beta}^{(1)}&=&(c^{(1)}+2m-1,\lambda_2^{m-1}+\beta_2^1,\ldots,\lambda_n^{m-1}+\beta_n^1) \in S \end{eqnarray*}
and since $ \lambda_h^j+\beta_h^1 >\lambda_h^0+\alpha_h^1$ for all $j=0,\ldots,m-1$ and $h=2,\ldots,n$, we can use the same strategy to show that 
$$  (c^{(1)}+m+y,\lambda_2^0+\alpha_2^1,\ldots,\lambda_n^0+\alpha_n^1) \in S,$$ 
for all $y=0,\ldots,2m-1$. Now it is clear that we can endlessly repeat the strategy and we finally proved that 
$$  (c^{(1)}+m+y,\lambda_2^0+\alpha_2^1,\ldots,\lambda_n^0+\alpha_n^1) \in S,$$ 
for all $y \in \mathbb{N}$ and for all the $S \in S_{\eta}$ ($S$ was arbitrarily chosen).
Therefore we proved that if $S \in S_{\eta}$, then the conductor of $S$ is smaller than $c_{\eta}$. Now we know that a good semigroup is completely characterized by its small elements. This implies that the set of good semigroups with  a conductor smaller than  $c_{\eta}$ is finite and therefore also $S_{\eta}$ must be finite.
\end{proof}
\begin{oss}
Let us consider a set of vector $\eta\subseteq \N^n$ which satisfies the hypothesis of the previous theorem. The proof of the theorem gives us also a way to determine a bound for the conductor of all good semigroups containing $\eta$. 
\end{oss}
\begin{defi}
\label{E11}
Given a good semigroup $S \subseteq \mathbb{N}^2$ and a set of vector $\eta\subseteq I_A(S)$, we say that $\eta$ is a \emph{system of representatives} of $S$, or more simply \emph{sor}, if $S\in S_{\eta}$. 
\end{defi}
\begin{oss}
\label{E12}
As a consequence of the Theorem \ref{E7}, $I_A(S)$ is a \emph{sor} of $S$, because every semigroup containing the elements of $I_A(S)$ must contain $S$.
\end{oss}

\begin{defi}
\label{E13}
A system of representatives $\eta$ of $S$ is minimal, if given another set of representatives $\eta'\subseteq \eta$, it follows $\eta'=\eta$. We call such a set a  \emph{msor} of $S$.
\end{defi}
It is possible to show that two \emph{msor} can have different cardinalities (see Example \ref{dmsor}).
\begin{defi}
\label{E14}
Given a good semigroup $S$, we define embedding dimension of $S$:
$$\operatorname{edim}(S)=\min\{|\eta|:S\in S_{\eta} \textrm{ and } \eta \subseteq I_A(S)\}.$$
\end{defi}

From this point onwards we will start to analyze the properties of the embedding dimension. We will consider only good semigroups $S\subseteq \N^2$.
Computing all the minimal good semigroups containing a set of vectors is computationally very dispensing, also in the two-branches case. At this point, our first aim is to produce a "fast" algorithm that, in the case of good semigroup $S\subseteq \N^2$, returns a \emph{msor} of $S$. In order to do this we will calculate two bounds for the embedding dimension.

\subsection{An inferior bound for the embedding dimension}
\label{section21}
First of all we want to produce an inferior bound for the embedding dimension. We give the following definitions.

\begin{defi}
\label{pitrack}
Given $\bs{\alpha},\bs{\beta}\in I_A(S)$ we say that $\bs{\alpha}$ and $\bs{\beta}$ are connected by a \emph{piece of track} if they are not comparable, i.e. $\bs{\alpha} \not \leq \bs{\beta}$ and $\bs{\beta} \not \leq \bs{\alpha}$, and denoted by $\bs{\gamma}=\bs{\alpha} \oplus \bs{\beta}$, we have $\Delta^S(\bs{\gamma})\cap (S\setminus I(S))=\emptyset$.
\end{defi}

\begin{defi}
\label{track}
Given $\bs{\alpha_1},\ldots,\bs{\alpha_n}\in I_A(S)$, with $\alpha_{11}<\ldots<\alpha_{n1}$ we say that $\bs{\alpha_1},\ldots,\bs{\alpha_n}$ are connected by a \emph{track} if 
we have:
\begin{itemize}
    \item $_2\Delta^S(\bs{\alpha_1})\cap(S\setminus I(S))=\emptyset$;
    \item $_1\Delta^S(\bs{\alpha_n})\cap(S\setminus I(S))=\emptyset$;
    \item $\bs{\alpha_i}$ and $\bs{\alpha_{i+1}}$ are connected by a \emph{piece of track} for all $i\in\{1,\ldots, n-1\}$.
\end{itemize}
In this case, denoted with $\bs{\gamma_i}=\bs{\alpha_i}\oplus \bs{\alpha_{i+1}}$ for $i\in\{1,\ldots, n-1\}$, we set:
$$T((\bs{\alpha_1},\ldots,\bs{\alpha_n}))=\{\bs{\alpha_1}\} \cup {}_2\Delta^S(\bs{\alpha_1}) \cup \left(\cup_{i=1}^{n-1}\Delta^S(\bs{\gamma_i})\right)\cup {}_1\Delta^S(\bs{\alpha_n}) \cup \{\bs{\alpha_n}\},$$
and we call this set  the \emph{track} connecting $\bs{\alpha_1},\ldots,\bs{\alpha_n}$.
\end{defi}
We will simply say that $T\subseteq S$ is  a \emph{track} in $S$ if there exist $\bs{\alpha_1},\ldots,\bs{\alpha_n}\in I_A(S)$ such that $T$ is the track connecting $\bs{\alpha_1},\ldots,\bs{\alpha_n}$.
Notice that  the previous definition implies that a track $T$ of $S$ never contains elements $\bs{\alpha}$ such that $\bs{\alpha} \geq \bs{c(S)}+\bs{e(S)}.$\\

In the following lemma we will show how these definitions are related to the embedding dimension.

\begin{lem}
\label{semcont}
Given a good semigroup $S$, and  a track $T=T((\bs{\alpha_1},\ldots,\bs{\alpha_n}))$ in $S$, then, 
$S'=S\setminus T$ is a good semigroup strictly contained in $S$. 
\end{lem}
\begin{proof}
If $\bs{\alpha},\bs{\beta}\in S'$, since
$\bs{\alpha},\bs{\beta}\in S$ and $T\cap (S\setminus I(S))=\emptyset$, we have $\bs{\alpha}+\bs{\beta}\in S'$, thus $S'$ is a semigroup. Now, we have to check that $S'$ satisfies the property (G1); therefore, considering $\bs{\alpha},\bs{\beta}\in S'$, we have to prove that $\bs{\alpha} \oplus \bs{\beta} \in S'$. If we suppose 
$\bs{\alpha}\oplus \bs{\beta}\in T$, then: there exists a $\bs{\gamma_i}=\bs{\alpha_i} \oplus \bs{\alpha_{i+1}}$ such that 
$\bs{\alpha} \oplus \bs{\beta}\in \Delta^S(\bs{\gamma_i})$; or $\bs{\alpha}\oplus \bs{\beta}\in {}_1\Delta^S(\bs{\alpha_1})$; or $\bs{\alpha}\oplus \bs{\beta}\in {}_2\Delta^S(\bs{\alpha_n})$. But, in all the previous cases, by the definition of track, 
this would imply that $\bs{\alpha},\bs{\beta}\in T$. Furthermore, for all $\bs{\alpha}\in S$ with $\bs{\alpha}\geq 
\bs{c(S)}+\bs{e(S)}$ we have $\bs{\alpha}\in S'$, thus $S'$ satisfies property (G2). We complete the proof verifying 
 the property (G3). Therefore, we take $\bs{\alpha},\bs{\beta}\in S'$ and suppose that $\bs{\beta}\in 
\Delta_i^{S'}(\bs{\alpha})$, we need to show that $\Delta_j^{S'}(\bs{\alpha})\neq \emptyset$, where 
$j\in\{1,2\}\setminus\{i\}$. Since $\bs{\alpha},\bs{\beta}\in S$, for property (G3), there 
exists $\bs{\delta}\in \Delta_j^S(\bs{\alpha})$. If $\bs{\delta}\in \Delta_j^{S'}(\bs{\alpha})$, the 
thesis is proved; hence we suppose the converse, in this case $\bs{\delta}$ necessarily belongs to 
$T'$. We have two cases. Case 1: there exists $\bs{\gamma_k}=\bs{\alpha_k}\oplus \bs{\alpha_{k+1}}$ such that 
$\bs{\delta}\in \Delta_j^S(\bs{\gamma_k})$, but this implies $\bs{\gamma_k}\in 
\Delta_j^{S'}(\bs{\alpha})$. Case 2: there exists $\bs{\gamma_k}=\bs{\alpha_k}\oplus \bs{\alpha_{k+1}}$ such that 
$\bs{\delta}\in \Delta_i^S(\bs{\gamma_k})$. We notice that, if $\bs{\delta}\in I_A(S)$ we can reduce to the
previous case, hence we can suppose that there exists $\bs{\rho}\neq \bs{\delta}$ with $\bs{\rho}\in 
\Delta_i^S(\bs{\gamma_k})\cap I_A(S)$. But, since $\bs{\rho}, \bs{\delta}\in S$, by property (G3) 
in $S$ and by  definition of track, $\Delta_j^{S'}(\bs{\delta})\neq \emptyset$, then we also  have
$\Delta_j^{S'}(\bs{\alpha})\neq \emptyset$. Case 3: $\bs{\delta}\in \Delta_i^S(\bs{\alpha_1})$ if $i=2$ or $\bs{\delta}\in \Delta_i^S(\bs{\alpha_n})$ if $i=1$; in this case we can conclude the proof with the same argument of Case 2.
\end{proof}

\begin{defi}
\label{hset}
Given $M\subset I_A(S)$, we say that $M$ is an \emph{hitting set} (HS) of $S$, if for any track $T$ in $S$ there exists an element $\bs{\alpha}\in M$ such that $\bs{\alpha}\in T$. 
We say that $M$ is a minimal hitting set (MHS), if for any hitting set $M$ such that $M'\subseteq M$, we have $M'=M$.
\end{defi}

\begin{oss}
\label{hs}
Given a hypergraph $(V,E)$, with $E=\{E_1,\ldots E_n\}$, $E_i\subseteq V$,  a set of vertices $H\subset V$ such that $H\cap E_i \neq \emptyset$ for all $i=1,\ldots,n$ is called \emph{transversal} or \emph{hitting set} \cite{Berge:Hypergraphs}.\\
If we consider the hypergraph with vertices $V=I_A(S)\subset \Gamma_S$ and edges $E=\{T\subset S\text{ : }$ $T \text{ is a track}\}$, then the hitting sets of the good semigroup $S$  correspond exactly to the hitting sets of this hypergraph.
The problem of finding the minimal hitting set of an hypergraph is an NP-hard problem and there are several algorithms related  to its computation (see for example \cite{Eiter:Hypergraph}, \cite{Murakami:hs}).
\end{oss}
We set $\mathfrak{H}=\{M\hspace{0.1cm}|\hspace{0.1cm}M \text{ is a HS}\}$.
\begin{prop}
\label{sorhit}
If $M$ is a sor, then $M\in \mathfrak{H}$.
\end{prop}
\begin{proof}
If we suppose that $M$ is not a HS, then it would  exist a track $T$ in $S$ that does not contain elements of $M$. Using the same construction of Lemma \ref{semcont} we could  build a good semigroup $S'$ such that $M\subseteq S'  \subsetneq S$, but it is a contradiction.
\end{proof}
The converse of the previous theorem  is not true in general as it is shown by the following example.
\begin{ex} \label{hsnosor}
Let us consider the good semigroup $S$ with the following set of irreducible absolute elements:
\begin{eqnarray*}I_A(S)&=&\{  (6, 3) , (12, 17), (18, 25), (19, 6), (24, \infty), (25, 28),  (27, 9), (31, \infty), (33, 20), \\
  && (39,\infty), (41, \infty), (44, \infty), (46, \infty), (\infty, 15), (\infty, 23),
  (\infty,31) \}. \end{eqnarray*}
 
  From Figure \ref{Fig2} we can easily deduce that $S$ contains only the following tracks:
  \begin{itemize}
      \item $T_1=T((6,3))$;
      \item $T_2=T((12,17),(19,6))$;
      \item $T_3=T((39,\infty),(\infty,31))$;
       \item $T_4=T((41,\infty),(\infty,23))$;
        \item $T_5=T((41,\infty),(\infty,31))$;
         \item $T_6=T((41,\infty))$;
          \item $T_7=T((46,\infty),(\infty,15))$;
           \item $T_8=T((46,\infty),(\infty,23))$;
            \item $T_9=T((46,\infty),(\infty,31))$.
  \end{itemize}
 
 \begin{figure}[!h]
  
    \centering
    \tikzset{mark size=1}\begin{tikzpicture}[scale=1.5]\begin{axis}[grid=major, xmin=0, ymin=0, xmax=57, ymax=42, ytick={0,3,6,9,15,17,20,23,25,28,29,31,35}, xtick={0,6,12,18,24,27,31,33,39,41,44,47}, yticklabel style={font=\tiny}, xticklabel style={font=\tiny}]\addplot[->, style=dotted, very thick]coordinates{(47.47474747474747,15)(57,15)};\addplot[->, style=dotted, very thick]coordinates{(47.47474747474747,18)(57,18)};\addplot[->, style=dotted, very thick]coordinates{(47.47474747474747,21)(57,21)};\addplot[->, style=dotted, very thick]coordinates{(47.47474747474747,23)(57,23)};\addplot[->, style=dotted, very thick]coordinates{(47.47474747474747,24)(57,24)};\addplot[->, style=dotted, very thick]coordinates{(47.47474747474747,26)(57,26)};\addplot[->, style=dotted, very thick]coordinates{(47.47474747474747,27)(57,27)};\addplot[->, style=dotted, very thick]coordinates{(47.47474747474747,29)(57,29)};\addplot[->, style=dotted, very thick]coordinates{(47.47474747474747,30)(57,30)};\addplot[->, style=dotted, very thick]coordinates{(47.47474747474747,31)(57,31)};\addplot[->, style=dotted, very thick]coordinates{(47.47474747474747,32)(57,32)};\addplot[->, style=dotted, very thick]coordinates{(47.47474747474747,33)(57,33)};\addplot[->, style=dotted, very thick]coordinates{(47.47474747474747,34)(57,34)};\addplot[->, style=dotted, very thick]coordinates{(47.47474747474747,35)(57,35)};\addplot[->, style=dotted, very thick]coordinates{(24,35.46666666666667)(24,42)};\addplot[->, style=dotted, very thick]coordinates{(30,35.46666666666667)(30,42)};\addplot[->, style=dotted, very thick]coordinates{(31,35.46666666666667)(31,42)};\addplot[->, style=dotted, very thick]coordinates{(36,35.46666666666667)(36,42)};\addplot[->, style=dotted, very thick]coordinates{(37,35.46666666666667)(37,42)};\addplot[->, style=dotted, very thick]coordinates{(39,35.46666666666667)(39,42)};\addplot[->, style=dotted, very thick]coordinates{(41,35.46666666666667)(41,42)};\addplot[->, style=dotted, very thick]coordinates{(42,35.46666666666667)(42,42)};\addplot[->, style=dotted, very thick]coordinates{(43,35.46666666666667)(43,42)};\addplot[->, style=dotted, very thick]coordinates{(44,35.46666666666667)(44,42)};\addplot[->, style=dotted, very thick]coordinates{(45,35.46666666666667)(45,42)};\addplot[->, style=dotted, very thick]coordinates{(46,35.46666666666667)(46,42)};\addplot[->, style=dotted, very thick]coordinates{(47,35.46666666666667)(47,42)};\addplot [pattern = north east lines, draw=white]coordinates{(47,35)(57,35)(57,42)(47,42)(47,35)};\addplot[only marks] coordinates{(0,0)(6,3)(12,9)(12,12)(12,15)(12,17)(18,6)(18,21)(18,23)(18,24)(18,25)(19,6)(24,20)(24,31)(24,33)(24,35)(25,12)(25,15)(25,18)(25,20)(25,21)(25,23)(25,24)(25,26)(25,27)(25,28)(27,9)(30,20)(31,20)(31,32)(31,33)(31,34)(31,35)(33,15)(33,18)(33,20)(39,27)(39,29)(39,30)(39,31)(39,32)(39,33)(39,34)(39,35)(41,15)(41,18)(41,21)(41,23)(41,24)(41,26)(41,27)(41,29)(41,30)(41,31)(41,32)(41,33)(41,34)(41,35)(42,23)(42,31)(43,23)(43,31)(44,23)(44,31)(44,35)(45,23)(45,31)(46,18)(46,21)(46,23)(46,24)(46,26)(46,27)(46,29)(46,30)(46,31)(46,32)(46,33)(46,34)(46,35)(47,15)(47,23)(47,31)};\addplot[only marks,mark=o, mark options={scale=2.7}] coordinates{(6,3)(12,17)(18,25)(19,6)(24,35)(25,28)(27,9)(31,35)(33,20)(39,35)(41,35)(44,35)(46,35)(47,15)(47,23)(47,31)};\addplot[only marks, mark=o] coordinates{(12,6)(18,9)(18,12)(18,15)(18,18)(18,20)(24,9)(24,12)(24,15)(24,18)(24,21)(24,23)(24,24)(24,26)(24,27)(24,28)(24,29)(24,30)(24,32)(24,34)(25,9)(30,12)(30,15)(30,18)(30,21)(30,23)(30,24)(30,26)(30,27)(30,29)(30,30)(30,31)(30,32)(30,33)(30,34)(30,35)(31,12)(31,15)(31,18)(31,21)(31,23)(31,24)(31,26)(31,27)(31,29)(31,30)(31,31)(33,12)(36,12)(36,15)(36,18)(36,21)(36,23)(36,24)(36,26)(36,27)(36,29)(36,30)(36,31)(36,32)(36,33)(36,34)(36,35)(37,12)(37,15)(37,18)(37,21)(37,23)(37,24)(37,26)(37,27)(37,29)(37,30)(37,31)(37,32)(37,33)(37,34)(37,35)(38,12)(39,15)(39,18)(39,21)(39,23)(39,24)(39,26)(42,15)(42,18)(42,21)(42,24)(42,26)(42,27)(42,29)(42,30)(42,32)(42,33)(42,34)(42,35)(43,15)(43,18)(43,21)(43,24)(43,26)(43,27)(43,29)(43,30)(43,32)(43,33)(43,34)(43,35)(44,15)(44,18)(44,21)(44,24)(44,26)(44,27)(44,29)(44,30)(44,32)(44,33)(44,34)(45,15)(45,18)(45,21)(45,24)(45,26)(45,27)(45,29)(45,30)(45,32)(45,33)(45,34)(45,35)(46,15)(47,18)(47,21)(47,24)(47,26)(47,27)(47,29)(47,30)(47,32)(47,33)(47,34)(47,35)};\end{axis}\end{tikzpicture}
   
    	\caption{$\bigcirc$: Irreducible Absolutes; $\circ$: Reducible Elements}
    	\label{Fig2}
\end{figure}
 
  Then, it is easy to verify that $M=\{ (6, 3), (12, 17), (39, \infty), (41, \infty), (46, \infty)\}$ is a MHS for $S$.
  
  However, $M$ is not a sor for $S$, in fact it is possible to check that there exists a good semigroup $S'$ with
 \begin{eqnarray*}
    I_A(S')&=&\{(6, 3), (12, 17), (19, 6), (24, \infty), (39, \infty), (41,\infty), (46, \infty), \\ &&
  (50, \infty),(\infty, 18), ( \infty, 29), (\infty, 34)  \},  
 \end{eqnarray*}
 such that $S'$ is strictly contained in $S$ and we have $M\subseteq S'$.
\end{ex}
Now we define: $\operatorname{bedim}(S) = \min\{|M|,\hspace{0.1cm}M\in \mathfrak{H}\}$.
\begin{cor} \label{infb}
Given a good semigroup $S\subseteq \N^2$, $\operatorname{bedim}(S)\leq \operatorname{edim}(S)$.
\end{cor}
\begin{ex}
The inequality of Corollary \ref{infb} can be strict. In fact, for instance, it is possible to check that each minimal hitting set of the semigroup $S$ described in Example \ref{hsnosor} is not a \emph{sor} for $S$, implying that $\operatorname{bedim}(S)< \operatorname{edim}(S)$.
\end{ex}

\subsection{A superior bound for the embedding dimension}
\label{section22}

Let  $S\subseteq \N^2$ be a good semigroup; given $\eta\subseteq I_A(S)$, and $\bs{\alpha}\in I_A(S)$, we want to define the reducibility of $\bs{\alpha}$ with respect to $\eta$. By convention we will say that all the elements $\bs{\alpha}\in \eta$ are reducible by $\eta$.
We take $\bs{\alpha}\in I_A(S)\backslash \eta$ and we will treat the finite and infinite elements separately.\\
\textbf{Finite case}:
We suppose $\bs{\alpha}=(\alpha_1,\alpha_2)\in I_{A_f}(S)\backslash{\eta}$.\\
Given a semiring $\Gamma \subseteq \overline{\mathbb{N}}^2$, we introduce the following notations:
\begin{eqnarray*}
 {}_i\Delta^{\Gamma}(\bs{\alpha})&:=&\Gamma \cap {}_i\Delta(\bs{\alpha})\\
		{}_1\delta^{\Gamma}(\bs{\alpha})&:=&\max\{y|(a,y)\in  {}_1\Delta^{\Gamma}(\bs{\alpha})\} \textrm{ if } {}_1\Delta^{\Gamma}(\bs{\alpha})\neq \emptyset \\
		{}_2\delta^{\Gamma}(\bs{\alpha})&:=&\max\{x|(x,b)\in  {}_2\Delta^{\Gamma}(\bs{\alpha})\}\textrm{ if } {}_2\Delta^{\Gamma}(\bs{\alpha})\neq \emptyset .
\end{eqnarray*}
Notice that the fact that $\bs{\alpha}$ is an absolute finite  element implies that  ${}_i\delta^{\Gamma}(\bs{\alpha})$ is finite. 
In the following, given $\eta \subseteq I_A(S)$, we will work with the semiring $\sgen{\eta}$. In order to simplify the notation we will write $ {}_i\Delta^{\eta}(\bs{\alpha})$ instead of ${}_i\Delta^{\sgen{\eta}}(\bs{\alpha})$ 
\begin{oss}
\label{ineq}
 If $ {}_i\Delta^{\eta}(\bs{\alpha})\neq \emptyset$, we have ${}_i\delta^{\eta}(\bs{\alpha})\leq {}_i\delta^{S}(\bs{\alpha})$.
\end{oss}

If$ {}_1\Delta^{\eta}(\bs{\alpha})\neq \emptyset$ , we define $Y^{\eta}(\bs{\alpha})=\{y\in\{{}_1\delta^{\eta}(\bs{\alpha}),\ldots,{}_1\delta^{S}(\bs{\alpha})\}|(\alpha_1,y)\in S\}$ and similarly if $ {}_2\Delta^{\eta}(\bs{\alpha})\neq \emptyset$ , we define $X^{\eta}(\bs{\alpha})=\{x\in\{{}_2\delta^{\eta}(\bs{\alpha}),\ldots,{}_2\delta^{S}(\bs{\alpha})\}|(x,\alpha_2)\in S\}$ .

\begin{defi}
We say that $\bs{\alpha}=(\alpha_1,\alpha_2)\in I_{A_f}(S)\backslash{\eta}$ is reducible by $\eta$, if ${}_1\Delta^{\eta}(\bs{\alpha})\cup {}_2\Delta^{\eta}(\bs{\alpha})\neq \emptyset$ and one of the following conditions is satisfied:
\label{defind}
\begin{enumerate}
	\item ${}_1\Delta^{\eta}(\bs{\alpha})\neq \emptyset$, and for all $y\in Y^{\eta}(\bs{\alpha})$, there exists $(x,y)\in \sgen{\eta}$ such that $x>\alpha_1$.
	\item ${}_2\Delta^{\eta}(\bs{\alpha})\neq \emptyset$, and for all $x\in X^{\eta}(\bs{\alpha})$, there exists $(x,y)\in \sgen{\eta}$ such that $y>\alpha_2$.

\end{enumerate}
\end{defi}

\textbf{Infinite case:}
If $\bs{\alpha}=(\alpha_1,\infty)\in I_{A}(S)^{\infty}\backslash{\eta}$, then we consider $\tilde{y}$  such that $(\alpha_1,y)\in S$ for all $y\geq \tilde{y}$ (it exists by Lemma \ref{E2}). Let us consider the set: $$Y^{\eta}(\bs{\alpha})=\{y\in \{{}_1\delta^{\eta}(\bs{\alpha}),\ldots,\max\{\tilde{y},{}_1\delta^{\eta}(\bs{\alpha})\}+e_2-1\}\hspace{0.1cm}|\hspace{0.1cm} (\alpha_1,y)\in S\}.$$
If $\bs{\alpha}=(\infty,\alpha_2)\in I_{A}(S)^{\infty}\backslash{\eta}$, then we consider $\tilde{x}$  such that $(x,\alpha_2)\in S$ for all $x\geq \tilde{x}$ (it exists by Lemma \ref{E2}). Let us consider the set: $$X^{\eta}(\bs{\alpha})=\{x\in \{{}_2\delta^{\eta}(\bs{\alpha}),\ldots,\max\{\tilde{x},{}_2\delta^{\eta}(\bs{\alpha})\}+e_1-1\}\hspace{0.1cm}|\hspace{0.1cm} (x,\alpha_2)\in S\}.$$
\begin{defi}
\label{E19}
We say that $\bs{\alpha}=(\alpha_1,\infty)$ is reducible by $\eta$, if ${}_1\Delta^{\eta}(\bs{\alpha})\neq \emptyset$ and for all $y\in Y^{\eta}(\bs{\alpha})$, there exists an element $(x,y)\in \sgen{\eta}$ with $x>\alpha_1$.
We say that $\bs{\alpha}=(\infty,\alpha_2)$ is reducible by $\eta$, if ${}_2\Delta^{\eta}(\bs{\alpha})\neq \emptyset$ and for all $x\in X^{\eta}(\bs{\alpha})$, there exists an element $(x,y)\in \sgen{\eta}$ with $y>\alpha_2$.
\end{defi}
As we will see in detail in the proof of Theorem \ref{E21}, the previous definitions are motivated by the fact  that the reducibility of an  element $\bs{\alpha} \in I_A(S)$ by a set $\eta \subseteq I_A(S)$ essentially ensures that the presence of $\bs{\alpha}$ in $I_A(S)$ is forced by $\eta$ as a consequence of property (G3) of good semigroups.

Given $\eta\subseteq I_A(S)$, we set:
$$\red{\eta}:=\{\bs{\alpha}\in I_A(S)\hspace{0.1cm}|\hspace{0.1cm} \bs{\alpha} \text{ is reducible by } \eta\}.$$ 

Let us consider the following algorithm:
\begin{algorithm}
	\SetKwData{Left}{left}
	\SetKwData{This}{this}
	\SetKwData{Up}{up}
	\SetKwFunction{Union}{Union}
	\SetKwFunction{FindCompress}{FindCompress}
	\SetKwInOut{Input}{input}
	\SetKwInOut{Output}{output}
	\Input{$\eta\subseteq I_A(S)$}
	\Output{A subset $\eta'$, with $\eta\subseteq \eta'\subseteq I_A$}
	\BlankLine
	$\eta' \longleftarrow \langle \langle \eta\rangle \rangle$\\
	\While{$\eta'\neq \eta$} {
	{$\eta\longleftarrow \eta'$}\\
	{$\eta'\longleftarrow \langle \langle \eta\rangle \rangle$}
	}
	\Return $ \eta'$
	\label{algo1}
	\caption{Algorithm to find $\operatorname{red}(\eta)$}
\end{algorithm} 

The input of Algorithm \ref{algo1} is a subset $\eta$ of $I_A(S)$. As long as we can, we expand $\eta$ by including elements of $I_A(S) \setminus{\eta}$ that are reducible by it. Notice that the algorithm produces an output in finite time, since $I_A(S)$ is finite.
We denote by $\operatorname{red}(\eta)$ the output of the previous algorithm and we introduce the set $\mathfrak{R}(S)=\{\eta \subseteq I_A(S)\hspace{0.1cm}|\hspace{0.1cm}\operatorname{red}(\eta)=I_A(S)\}$. We will say that $\eta\subseteq I_A(S)$ satisfy the \emph{reducibility condition} if $\eta \in \mathfrak{R}(S)$.\\
We have the following statement:
\begin{teo} 
\label{E21}
If $\eta\in \mathfrak{R}(S)$, then $\eta$ is a sor.
\end{teo}
\begin{proof}
 From $\eta\in  \mathfrak{R}(S)$ it follows that there exists a chain of subset of $I_A(S)$:
$$\eta\subset \eta_1 \subset \ldots \subset \eta_{n-1}\subset \eta_n=\operatorname{red}(\eta)=I_A(S)$$
such that $\eta_i=\red{\eta_{i-1}}$
We prove that $\eta$ is a \emph{sor} using a decreasing induction on this chain. We have that $\eta_n=I_A(S)$ is a \emph{sor} for Remark \ref{E12}, now we prove that if $\eta_{i+1}$ is a \emph{sor}, then $\eta_{i}$ is a \emph{sor}.\\
We assume by contradiction that $S\notin S_{\eta_{i}}$; in this case there exists a good semigroup $S_{i}$ such that $\eta_i\subseteq S_i\subsetneq S$.\\
If we suppose $\eta_{i+1}\subseteq I_A(S_i)$, we would have $\eta_{i+1}\subseteq \sgen{\eta_{i+1}} \subseteq\sgen{I_A(S_{i})}=S_i\subsetneq S$, against the fact that $\eta_{i+1}$ is a \emph{sor} for $S$.
For this reason, we can always suppose that there exists $\bs{\alpha}=(\alpha_1,\alpha_2)\in \eta_{i+1} \backslash I_A(S_i)$. Furthermore we observe that $\bs{\alpha}\notin \eta_i$, indeed, assuming the opposite, we should have $\bs{\alpha}\in S_i$ and since $S_i\subseteq S$ and $\bs{\alpha}\in \eta_{i+1}\subseteq I_A(S)$, it would imply that $\bs{\alpha}\in I_A(S_i)$. We distinguish two case: $\bs{\alpha}\in \eta_{i+1}\cap I_{A_f}(S)$ and $\bs{\alpha}\in \eta_{i+1} \cap I_{A}^{\infty}(S)$.\\
\textbf{Case 1:} $\bs{\alpha}\in \eta_{i+1}\cap I_{A_f}(S)$.\\
Since $\red{\eta_i}=\eta_{i+1}$, $\bs{\alpha}$ is reducible by $\eta_i$. Without loss of generality we can assume ${}_1\Delta^{\eta_i}(\bs{\alpha})\neq \emptyset$; in this case there exists $(\alpha_1,{}_1\delta^{\eta_i}(\bs{\alpha}))\in \sgen{\eta_i}\subseteq S_i$.  We have ${}_1\delta^{\eta_i}(\bs{\alpha})\in Y^{\eta_1}(\bs{\alpha})$ and, from the reducibility of $\bs{\alpha}$ by $\eta_i$, there exists $(x^{\eta}(\bs{\alpha}),{}_1\delta^{\eta_i}(\bs{\alpha}))\in \sgen{\eta_i}\subseteq S_i$.
 We have obtained two elements $(\alpha_1,{}_1\delta^{\eta_i}(\bs{\alpha}))$, $(x^{\eta}(\bs{\alpha}),{}_1\delta^{\eta_i}(\bs{\alpha}))\in S_i$, by property (G3), there exists $(\alpha_1,y)\in S_i$, with $y>{}_1\delta^{\eta_i}(\bs{\alpha})$. We observe that, from the definition of ${}_1\delta^{S}(\bs{\alpha})$, $y\leq {}_1\delta^{S}(\bs{\alpha})$. Hence $y\in Y^{\eta_i}(\bs{\alpha})$. We can repeat the same argument until we obtain that $(\alpha_1,{}_1\delta^{S}(\bs{\alpha}))\in S_i$. Using again the property (G3) we should obtain $\bs{\alpha}\in S_i$ (notice that $\Delta_1^{S_i}(\bs{\alpha})=\emptyset$), but this is a contradiction.\\
 \textbf{Case 2:} $\bs{\alpha}\in \eta_{i+1}\cap I_{A^{\infty}}(S)$.\\ Without loss of generality we can suppose $\bs{\alpha}=(\alpha_1,\infty)$. Since $\bs{\alpha}$ is reducible by $\eta_i$, we have ${}_1\Delta^{\eta_i}(\bs{\alpha})\neq \emptyset$. We set $M(\bs{\alpha}):=\max\{\tilde{y},u\}+e_2-1$, where $\tilde{y}$ is such that $(\alpha_1,y)\in S$ for any $y>\tilde{y}$. Now, using the same argument of the finite case, we obtain that $(\alpha_1, M(\bs{\alpha}))\in S_i$, but, by Lemma \ref{E2}, this implies $(\alpha_1,\infty)\in S_i$ which is a contradiction.\\ 
\end{proof}
The following example shows that the converse of the previous theorem  is not true in general.
\begin{ex}
\label{sornotred}
Let us consider the good semigroup $S$ with the following set of irreducible absolute elements:
\begin{eqnarray*}I_A(S)&=&\{  (3, 4) , (6, \infty), (7, 8), (10, 15), (14,18), (17, 25),  (\infty, 12), (\infty,19), (\infty, 22),
  (\infty,29) \}. \end{eqnarray*}
  
\begin{figure}[!h]
	\begin{minipage}{0.5 \textwidth}
		\tikzset{mark size=1}\begin{tikzpicture}\begin{axis}[grid=major, xmin=0, ymin=0, xmax=26, ymax=36, ytick={0,4,8,12,15,18,19,22,25,26,29,30}, xtick={0,3,7,10,14,17,18,21}, yticklabel style={font=\tiny}, xticklabel style={font=\tiny}]\addplot[->, style=dotted, very thick]coordinates{(21.4468085106383,12)(26,12)};\addplot[->, style=dotted, very thick]coordinates{(21.4468085106383,16)(26,16)};\addplot[->, style=dotted, very thick]coordinates{(21.4468085106383,19)(26,19)};\addplot[->, style=dotted, very thick]coordinates{(21.4468085106383,20)(26,20)};\addplot[->, style=dotted, very thick]coordinates{(21.4468085106383,22)(26,22)};\addplot[->, style=dotted, very thick]coordinates{(21.4468085106383,23)(26,23)};\addplot[->, style=dotted, very thick]coordinates{(21.4468085106383,24)(26,24)};\addplot[->, style=dotted, very thick]coordinates{(21.4468085106383,26)(26,26)};\addplot[->, style=dotted, very thick]coordinates{(21.4468085106383,27)(26,27)};\addplot[->, style=dotted, very thick]coordinates{(21.4468085106383,28)(26,28)};\addplot[->, style=dotted, very thick]coordinates{(21.4468085106383,29)(26,29)};\addplot[->, style=dotted, very thick]coordinates{(21.4468085106383,30)(26,30)};\addplot[->, style=dotted, very thick]coordinates{(6,30.46153846153846)(6,36)};\addplot[->, style=dotted, very thick]coordinates{(9,30.46153846153846)(9,36)};\addplot[->, style=dotted, very thick]coordinates{(12,30.46153846153846)(12,36)};\addplot[->, style=dotted, very thick]coordinates{(13,30.46153846153846)(13,36)};\addplot[->, style=dotted, very thick]coordinates{(15,30.46153846153846)(15,36)};\addplot[->, style=dotted, very thick]coordinates{(16,30.46153846153846)(16,36)};\addplot[->, style=dotted, very thick]coordinates{(18,30.46153846153846)(18,36)};\addplot[->, style=dotted, very thick]coordinates{(19,30.46153846153846)(19,36)};\addplot[->, style=dotted, very thick]coordinates{(20,30.46153846153846)(20,36)};\addplot[->, style=dotted, very thick]coordinates{(21,30.46153846153846)(21,36)};\addplot [pattern = north east lines, draw=white]coordinates{(21,30)(26,30)(26,36)(21,36)(21,30)};\addplot[only marks] coordinates{(0,0)(3,4)(6,12)(6,15)(6,16)(6,18)(6,19)(6,20)(6,22)(6,23)(6,24)(6,25)(6,26)(6,27)(6,28)(6,29)(6,30)(7,8)(9,15)(9,18)(9,25)(10,15)(12,12)(12,18)(12,25)(13,12)(13,18)(13,22)(13,25)(13,29)(14,12)(14,18)(15,12)(15,19)(15,25)(16,12)(16,19)(16,25)(17,12)(17,19)(17,24)(17,25)(18,12)(18,19)(18,22)(19,12)(19,19)(19,22)(20,12)(20,19)(20,22)(21,12)(21,19)(21,22)(21,29)};\addplot[only marks,mark=o, mark options={scale=2.7}] coordinates{(3,4)(6,30)(7,8)(10,15)(14,18)(17,25)(21,12)(21,19)(21,22)(21,29)};\addplot[only marks, mark=o] coordinates{(6,8)(9,12)(9,16)(9,19)(9,20)(9,22)(9,23)(9,24)(9,26)(9,27)(9,28)(9,29)(9,30)(10,12)(12,16)(12,19)(12,20)(12,22)(12,23)(12,24)(12,26)(12,27)(12,28)(12,29)(12,30)(13,16)(13,19)(13,20)(13,23)(13,24)(13,26)(13,27)(13,28)(13,30)(14,16)(15,16)(15,20)(15,22)(15,23)(15,24)(15,26)(15,27)(15,28)(15,29)(15,30)(16,16)(16,20)(16,22)(16,23)(16,24)(16,26)(16,27)(16,28)(16,29)(16,30)(17,16)(17,20)(17,22)(17,23)(18,16)(18,20)(18,23)(18,24)(18,26)(18,27)(18,28)(18,29)(18,30)(19,16)(19,20)(19,23)(19,24)(19,26)(19,27)(19,28)(19,29)(19,30)(20,16)(20,20)(20,23)(20,24)(20,26)(20,27)(20,28)(20,29)(20,30)(21,16)(21,20)(21,23)(21,24)(21,26)(21,27)(21,28)(21,30)};\end{axis}\end{tikzpicture}		
	\end{minipage}
	\hspace{0.5cm}
	\begin{minipage}{0.5 \textwidth}
	\tikzset{mark size=1}\begin{tikzpicture}\begin{axis}[grid=major, xmin=0, ymin=0, xmax=26, ymax=36, ytick={0,4,8,12,15,18,19,22,25,26,29,30}, xtick={0,3,7,10,14,17,18,21}, yticklabel style={font=\tiny}, xticklabel style={font=\tiny}]\addplot[->, style=dotted, very thick]coordinates{(21.4468085106383,12)(26,12)};\addplot[->, style=dotted, very thick]coordinates{(21.4468085106383,16)(26,16)};\addplot[->, style=dotted, very thick]coordinates{(21.4468085106383,19)(26,19)};\addplot[->, style=dotted, very thick]coordinates{(21.4468085106383,20)(26,20)};\addplot[->, style=dotted, very thick]coordinates{(21.4468085106383,22)(26,22)};\addplot[->, style=dotted, very thick]coordinates{(21.4468085106383,23)(26,23)};\addplot[->, style=dotted, very thick]coordinates{(21.4468085106383,24)(26,24)};\addplot[->, style=dotted, very thick]coordinates{(21.4468085106383,26)(26,26)};\addplot[->, style=dotted, very thick]coordinates{(21.4468085106383,27)(26,27)};\addplot[->, style=dotted, very thick]coordinates{(21.4468085106383,28)(26,28)};\addplot[->, style=dotted, very thick]coordinates{(21.4468085106383,29)(26,29)};\addplot[->, style=dotted, very thick]coordinates{(21.4468085106383,30)(26,30)};\addplot[->, style=dotted, very thick]coordinates{(6,30.46153846153846)(6,36)};\addplot[->, style=dotted, very thick]coordinates{(9,30.46153846153846)(9,36)};\addplot[->, style=dotted, very thick]coordinates{(12,30.46153846153846)(12,36)};\addplot[->, style=dotted, very thick]coordinates{(13,30.46153846153846)(13,36)};\addplot[->, style=dotted, very thick]coordinates{(15,30.46153846153846)(15,36)};\addplot[->, style=dotted, very thick]coordinates{(16,30.46153846153846)(16,36)};\addplot[->, style=dotted, very thick]coordinates{(18,30.46153846153846)(18,36)};\addplot[->, style=dotted, very thick]coordinates{(19,30.46153846153846)(19,36)};\addplot[->, style=dotted, very thick]coordinates{(20,30.46153846153846)(20,36)};\addplot[->, style=dotted, very thick]coordinates{(21,30.46153846153846)(21,36)};\addplot [pattern = north east lines, draw=white]coordinates{(21,30)(26,30)(26,36)(21,36)(21,30)};\addplot[only marks] coordinates{(0,0)(6,12)(6,15)(6,16)(6,18)(6,19)(6,20)(6,22)(6,23)(6,24)(6,25)(6,26)(6,27)(6,28)(6,29)(6,30)(9,15)(9,18)(9,25)(12,18)(12,25)(13,22)(13,25)(13,29)(15,19)(15,25)(16,19)(16,25)(17,19)(17,24)(17,25)(18,19)(19,19)(20,19)(21,19)(21,29)};\addplot[only marks,mark=o, mark options={scale=2.7}] coordinates{(3,4)(6,30)(7,8)(10,15)(14,18)(17,25)(21,12)(21,19)(21,22)(21,29)};\addplot[only marks, mark=text, mark options={scale=1,text mark={\scalebox{.5}{$\eta$}}},text mark as node=true] coordinates{(6,8)(9,12)(10,12)(12,16)(13,16)(13,19)(14,16)(15,16)(15,20)(16,16)(16,20)(16,22)(16,23)(17,16)(17,20)(17,22)(17,23)(18,16)(18,20)(18,24)(19,16)(19,20)(19,24)(19,26)(19,27)(20,16)(20,20)(20,24)(20,26)(20,27)(20,28)(20,30)(21,16)(21,20)(21,24)(21,26)(21,27)(21,28)(21,30)};\addplot[only marks, mark=text, mark options={scale=1,text mark={\scalebox{.5}{$\eta$}}}, text mark as node=true] coordinates{(3,4)(7,8)(10,15)(12,12)(13,12)(13,18)(14,12)(14,18)(15,12)(16,12)(17,12)(18,12)(18,22)(19,12)(19,22)(20,12)(20,22)(21,12)(21,22)};\addplot[only marks] coordinates{(9,16)(9,19)(9,20)(9,22)(9,23)(9,24)(9,26)(9,27)(9,28)(9,29)(9,30)(12,19)(12,20)(12,22)(12,23)(12,24)(12,26)(12,27)(12,28)(12,29)(12,30)(13,20)(13,23)(13,24)(13,26)(13,27)(13,28)(13,30)(15,22)(15,23)(15,24)(15,26)(15,27)(15,28)(15,29)(15,30)(16,24)(16,26)(16,27)(16,28)(16,29)(16,30)(18,23)(18,26)(18,27)(18,28)(18,29)(18,30)(19,23)(19,28)(19,29)(19,30)(20,23)(20,29)(21,23)};\end{axis}\end{tikzpicture}

	\end{minipage}
	\caption{$\circ$: Reducible elements; $\bigcirc$: Irreducible Absolutes; $\eta$: Elements of $\sgen{\operatorname{red}(\eta)}$}
	\label{Figure 2}
\end{figure}
  
  Notice that, since $S$ contains only the tracks $T_1=T((3,4)), T_2=T((6,\infty),(7,8)), T_3=T((6,\infty),(10,15),(\infty,12))$ and $T_4=T((10,15),(\infty,12))$, we have that $\eta=\{ (3,4),(7,8),$ $(10,15), (14,18), (\infty,12), (\infty,22) \} $ is a HS for $S$. Let us show that $\operatorname{red}(\eta) \neq I_A(S)$. It suffices to show that $\red{\eta}=\eta$, i.e. all the elements in $I_A(S) \setminus \eta$ are not reducible by $\eta$.
  We have \begin{itemize}
      \item $ \bs{\alpha_1}=(6,\infty)$ is not reducible by $\eta$. Notice that there exists $(6,8)=2(3,4) \in {}_1\Delta^{\eta}(\bs{\alpha_1})$, thus we have ${}_1\delta^{\eta}(\bs{\alpha_1})=8$. Furthermore, $\tilde{y}=22$ and we have:
      \begin{eqnarray*}
      Y^{\eta}(\bs{\alpha_1})&=&\{y \in \{{}_1\delta^{\eta}(\bs{\alpha_1}),\ldots,\max\{\tilde{y},{}_1\delta^{\eta}(\bs{\alpha_1})\}+e_2-1| (6,y) \in S\}=\\
      &&=\{8,12,15,16,18,19,20,22,23,24,25\}.
      \end{eqnarray*}
      
      For each element $y$ in $Y^{\eta}(\bs{\alpha_1})$ we need to find $(x,y) \in \sgen{\eta}$  with $x>6$.
      It is not difficult to notice that for $y=25\in Y^{\eta}(\bs{\alpha_1})$, it is not possible to produce such an element in $\sgen{\eta}$.
       \item $ \bs{\alpha_2}=(17,25)$ is not reducible by $\eta$. Notice that there exists $(17,23)=(7,8)\odot (10,15) \in {}_1\Delta^{\eta}(\bs{\alpha_2})$, thus we have ${}_1\delta^{\eta}(\bs{\alpha_2})=23$ (while ${}_2\Delta^{\eta}(\bs{\alpha_2})=\emptyset$). Furthermore, ${}_1\delta^{S}(\bs{\alpha_2})=24$ and we have:
      $$Y^{\eta}(\bs{\alpha_2})=\{y \in \{{}_1\delta^{\eta}(\bs{\alpha_2}),\ldots,{}_1\delta^{S}(\bs{\alpha_2})=24| (17,y) \in S\}=\{23,24\}.$$
      For each element $y$ in $Y^{\eta}(\bs{\alpha_2})$ we need to find $(x,y) \in \sgen{\eta}$  with $x>17$.
      However for $y=23\in Y^{\eta}(\bs{\alpha_2})$, it is not possible to produce such an element in $\sgen{\eta}$
            \item $ \bs{\alpha_3}=(\infty,19)$ is not reducible by $\eta$. Notice that there exists $(13,19)=(3,4)\odot(10,15) \in {}_2\Delta^{\eta}(\bs{\alpha_3})$, thus we have ${}_2\delta^{\eta}(\bs{\alpha_3})=13$. Furthermore, $\tilde{x}=15$ and we have:
      $$X^{\eta}(\bs{\alpha_3})=\{x \in \{{}_2\delta^{\eta}(\bs{\alpha_3}),\ldots,\max\{\tilde{x},{}_2\delta^{\eta}(\bs{\alpha_3})\}+e_1-1| (x,19) \in S\}=\{13,15,16,17\}.$$
      For each element $x$ in $X^{\eta}(\bs{\alpha_3})$ we need to find $(x,y) \in \sgen{\eta}$  with $y>19$.
      It is not difficult to notice that for $x=13\in X^{\eta}(\bs{\alpha_3})$, it is not possible to do that.
            \item $ \bs{\alpha_4}=(\infty,29)$ is not reducible by $\eta$, since ${}_2\Delta^{\eta}(\bs{\alpha_4})=\emptyset.$
  \end{itemize}
  However it is possible to check that there are no good semigroups $S'$ such that $\eta \subseteq S' \subsetneq S$. Thus $\eta$ is actually a sor for $S$ and it is not difficult to check that the minimal hitting set  
  $M=\{(3,4), (7,8), (10,15)\}$  contained in it, is a \emph{sor} itself, thus  a \emph{msor} for $S$.
\end{ex}

Now we define: $\operatorname{Bedim}(S)=\min\{|\eta|, \eta\in \mathfrak{R}(S)\}$, 
\begin{cor}
\label{E23}
Given a good semigroup $S\subseteq \N^2$, $\operatorname{edim}(S)\leq \operatorname{Bedim}(S)$.
\end{cor}
\begin{ex}
The inequality in Corollary \ref{E23} can be strict. An example of this behaviour is the good semigroup $S$ with the following set of irreducible absolute elements:
\begin{eqnarray*}I_A(S)&=&\{  (7, 7) , (14, 20), (17, 14),  (24, \infty), (25, 21),  (32, 30), (39, 45), 
  (42,\infty), (43, 35), (44,37),\\ && (46,\infty), (47,50),(50, \infty), (54, \infty), (\infty, 32), (\infty, 34),
  (\infty,42),(\infty, 51),
  (\infty,57) \}. \end{eqnarray*}
  It is possible to prove that for each MHS
$\eta$ of $S$ we have that $\eta$ is a sor for $S$ but $\operatorname{red}(\eta)\neq I_A(S).$ This easily implies that $\operatorname{edim}(S)< \operatorname{Bedim}(S).$ \end{ex}
\begin{ex} \label{dmsor}
Let us consider the good semigroup $S$, with 
\begin{eqnarray*}I_A(S)&=&\{  (4, 3), (6, 7), (8, 8), (9, 6), (11, \infty), (12, \infty), (13, \infty ), ( 14,\infty),
  ( \infty, 9 ), (\infty, 11), (\infty, 13) \}. \end{eqnarray*}
  This is an example of good semigroup having \emph{msor} with distinct cardinalities. In fact, it is possible to prove that the sets $\eta_1=\{ (4, 3), (6, 7), (8, 8), (11,\infty) , ( 13, \infty) \}$ and $\eta_2=\{ (4, 3), (6, 7), (8, 8),$ $(11,\infty) , (  \infty,9), (\infty,11) \}$
 are both MHS of $S$ satisfying the reducibility condition. In particular $\operatorname{edim}(S)=5$.\end{ex}
\subsection{An algorithm for the computation of the embedding dimension of a semigroup $S\subseteq \N^2$}
\label{section23}
We will conclude this section presenting an algorithm for the computation of the embedding dimension and with some remarks concerning the definition that we have given.

We proved that:
$$\operatorname{bedim}(S)\leq \operatorname{edim}(S)\leq\operatorname{Bedim}(S)$$
and both inequalities are sharp as we will see in Example \ref{E31}. 

We implemented in GAP \cite{GAP4} the following functions:

\begin{itemize}
    \item ComputeMHS($S$): it takes in input a good semigroup and returns the set of its MHS.\\
    \item VerifyReducibility(list): it takes in input a list of subsets of $I_A(S)$ and returns the first set that satisfy the condition of reducibility if there exists, otherwise it returns "fail".\\
    \item IsThereAMGSContainedInAndContaining($S$,$V$): it takes in input a good semigroup $S$ and a subset  $V$ of $I_A(S)$ and returns "true" if there exists a good semigroup $S'$ such that $V\subseteq S'\subsetneq S$\\
\end{itemize}

\begin{oss}
\label{speed}
Testing in GAP these functions on a sample of about $200000$ semigroups, we observed empirically that VerifyReducibility is about seventy times faster than IsThereAMGSContainedInAndContaining.
\end{oss}
We introduce the following algorithm to compute the embedding dimension and a set of representatives with minimal cardinality.\\

\begin{algorithm} 
	\SetKwData{Left}{left}
	\SetKwData{This}{this}
	\SetKwData{Up}{up}
	\SetKwFunction{Union}{Union}
	\SetKwFunction{FindCompress}{FindCompress}
	\SetKwInOut{Input}{input}
	\SetKwInOut{Output}{output}
	\Input{A good Semigroup $S$}
	\Output{A minimal system of representatives of minimal cardinality}
	\BlankLine
	$\mathfrak{M}\longleftarrow ComputeMHS(S)$\\
	$\mathfrak{H}\longleftarrow \mathfrak{M}$ \\
	$n\longleftarrow \operatorname{bedim}(S)$ \\
Stop $\longleftarrow $ false \\
	\While{Stop=false} 
	{
	$\mathfrak{H}\longleftarrow \{\eta \subseteq I_A(S)\hspace{0.1cm}|\hspace{0.1cm} |\eta|=n$ \textrm{ and } $H \subseteq \eta $ \textrm{ for some } $H\in\mathfrak{H} \} \cup \{ \eta \in \mathfrak{M} | |\eta|=n \}$\\
	\If{VerifyReducibility($\mathfrak{H}$)=$\eta$}
	{Stop=true, \Return $\eta$\\}
	\If{VerifyReducibility($\mathfrak{H}$)=fail}{\eIf
	{ForAny $\eta \in \mathfrak{H}$, IsThereAMGSContainedInAndContaining(S,$\eta$)=false}
	{Stop=true, \Return $\eta$\\}
	{$n\longleftarrow n+1$\\}
	}
	}
	\label{algo4}
	\caption{Algorithm to find an \emph{msor} of minimal cardinality}
\end{algorithm}

\begin{oss}
We tested the algorithm on a sample of $200000$ good semigroups and we noticed that, for $n=\operatorname{bedim}(S)$:
\begin{itemize}
\item The condition "\emph{VerifyReducibility}($\mathfrak{H})=fail$" occurred only in $82$ cases.
\item Both the conditions "\emph{VerifyReducibility}($\mathfrak{H})=fail$" and "\emph{IsThereAMGSContainedInAndContaining}$(S,\eta)=true$ for all $\eta \in \mathfrak{H}$" occurred only in $2$ cases. In these cases $\operatorname{bedim}(S)\neq \operatorname{edim}(S)$.
\item The situation which all MSH of minimal cardinality are not reducible and at least one of them is a \emph{sor} occurred only in one case. In this case $\operatorname{Bedim}(S)\neq \operatorname{edim}(S)$).
\end{itemize}
For this reasons and by Remark \ref{speed} this algorithm is considerably faster than to computing the embedding dimension using the definition. 
\end{oss}

\begin{ex} \label{E31}
\label{exalg}
	Let us consider the good semigroup $S$, represented in Figure \ref{algfig2}, we want to find a \emph{msor} for $S$ and the embedding dimension of $S$.

We have that $$I_A(S)= \{ (4, 3), (7, 13), (11, 17), (14, \infty), (15, \infty), (16, 20),  (24, \infty), (\infty, 12),
(\infty, 16), (\infty, 26) \}. $$

First of all we need to compute the minimal hitting sets of $S$. It contains the following tracks:
\begin{itemize}
    \item $T_1=T((4,3));$
    \item $T_2=T((7,13));$
    \item $T_3=T((11,17),(\infty,16));$
    \item $T_4=T((15,\infty),(16,20),(\infty,12));$
    \item $T_5=T((15,\infty),(16,20),(\infty,16));$
    \item $T_6=T((24,\infty),(\infty,26)).$
\end{itemize}
Thus the following is the complete list of the MHS of $S$.
\begin{itemize}
	\item $\eta_1=\{ ( 4, 3 ), ( 7, 13 ), ( \infty, 12 ), ( \infty, 16 ), ( \infty, 26 ) \};$
	\item $\eta_2=
	\{ ( 4, 3 ), ( 7, 13 ), ( 11, 17 ), ( 15, \infty ), ( \infty, 26 ) \}$;
	\item $\eta_3=\{ ( 4, 3 ), ( 7, 13 ), ( 11, 17 ), ( 16, 20 ), ( 24, \infty ) \}$;
	\item	$\eta_4=\{ ( 4, 3 ), ( 7, 13 ), ( 11, 17 ), ( 16, 20 ), ( \infty, 26 ) \}$;
	\item	$\eta_5=\{ ( 4, 3 ), ( 7, 13 ), ( 15, \infty ), ( 24, \infty ), ( \infty, 16 ) \}$;
	\item $\eta_6=\{ ( 4, 3 ), ( 7, 13 ), ( 15, \infty ), ( \infty, 16 ), ( \infty, 26 ) \}$;
	\item $	\eta_7=\{ ( 4, 3 ), ( 7, 13 ), ( 16, 20 ), ( 24, \infty ), ( \infty, 16 ) \}$;
	\item $\eta_8=\{ ( 4, 3 ), ( 7, 13 ), ( 16, 20 ), ( \infty, 16 ), ( \infty, 26 ) \}$;
	\item $\eta_9=\{ ( 4, 3 ), ( 7, 13 ), ( 24, \infty ), ( \infty, 12 ), ( \infty, 16 ) \}$;
	\item $\eta_{10}=\{ 	( 4, 3 ), ( 7, 13 ), ( 11, 17 ), ( 15, \infty ), ( 24, \infty )  \}$. 

\end{itemize}
\begin{figure}[!h]

	\begin{minipage}{0.5 \textwidth}
\tikzset{mark size=1}\begin{tikzpicture}\begin{axis}[grid=major, xmin=0, ymin=0, xmax=30, ymax=33, ytick={0,3,12,13,15,16,17,20,24,26,27}, xtick={0,4,7,11,14,16,21,25}, yticklabel style={font=\tiny}, xticklabel style={font=\tiny}]\addplot[->, style=dotted, very thick]coordinates{(25.45454545454545,12)(30,12)};\addplot[->, style=dotted, very thick]coordinates{(25.45454545454545,15)(30,15)};\addplot[->, style=dotted, very thick]coordinates{(25.45454545454545,16)(30,16)};\addplot[->, style=dotted, very thick]coordinates{(25.45454545454545,18)(30,18)};\addplot[->, style=dotted, very thick]coordinates{(25.45454545454545,19)(30,19)};\addplot[->, style=dotted, very thick]coordinates{(25.45454545454545,21)(30,21)};\addplot[->, style=dotted, very thick]coordinates{(25.45454545454545,22)(30,22)};\addplot[->, style=dotted, very thick]coordinates{(25.45454545454545,24)(30,24)};\addplot[->, style=dotted, very thick]coordinates{(25.45454545454545,25)(30,25)};\addplot[->, style=dotted, very thick]coordinates{(25.45454545454545,26)(30,26)};\addplot[->, style=dotted, very thick]coordinates{(25.45454545454545,27)(30,27)};\addplot[->, style=dotted, very thick]coordinates{(14,27.45762711864407)(14,33)};\addplot[->, style=dotted, very thick]coordinates{(15,27.45762711864407)(15,33)};\addplot[->, style=dotted, very thick]coordinates{(18,27.45762711864407)(18,33)};\addplot[->, style=dotted, very thick]coordinates{(19,27.45762711864407)(19,33)};\addplot[->, style=dotted, very thick]coordinates{(21,27.45762711864407)(21,33)};\addplot[->, style=dotted, very thick]coordinates{(22,27.45762711864407)(22,33)};\addplot[->, style=dotted, very thick]coordinates{(23,27.45762711864407)(23,33)};\addplot[->, style=dotted, very thick]coordinates{(24,27.45762711864407)(24,33)};\addplot[->, style=dotted, very thick]coordinates{(25,27.45762711864407)(25,33)};\addplot [pattern = north east lines, draw=white]coordinates{(25,27)(30,27)(30,33)(25,33)(25,27)};\addplot[only marks] coordinates{(0,0)(4,3)(7,6)(7,9)(7,12)(7,13)(11,17)(14,16)(14,20)(14,23)(14,27)(15,16)(15,21)(15,22)(15,23)(15,24)(15,25)(15,26)(15,27)(16,15)(16,16)(16,18)(16,19)(16,20)(18,12)(18,16)(19,12)(19,16)(20,12)(20,16)(21,12)(21,15)(21,16)(21,19)(22,12)(22,16)(23,12)(23,16)(24,12)(24,16)(24,27)(25,12)(25,16)(25,26)};\addplot[only marks,mark=o, mark options={scale=2.7}] coordinates{(4,3)(7,13)(11,17)(14,27)(15,27)(16,20)(24,27)(25,12)(25,16)(25,26)};\addplot[only marks, mark=o] coordinates{(8,6)(11,9)(11,12)(11,15)(11,16)(12,9)(14,12)(14,15)(14,18)(14,19)(14,21)(14,22)(14,24)(14,25)(14,26)(15,12)(15,15)(15,18)(15,19)(15,20)(16,12)(18,15)(18,18)(18,19)(18,21)(18,22)(18,23)(18,24)(18,25)(18,26)(18,27)(19,15)(19,18)(19,19)(19,21)(19,22)(19,23)(19,24)(19,25)(19,26)(19,27)(20,15)(20,18)(20,19)(20,21)(20,22)(20,23)(21,18)(21,21)(21,22)(21,24)(21,25)(21,26)(21,27)(22,15)(22,18)(22,19)(22,21)(22,22)(22,24)(22,25)(22,26)(22,27)(23,15)(23,18)(23,19)(23,21)(23,22)(23,24)(23,25)(23,26)(23,27)(24,15)(24,18)(24,19)(24,21)(24,22)(24,24)(24,25)(24,26)(25,15)(25,18)(25,19)(25,21)(25,22)(25,24)(25,25)(25,27)};\end{axis}\end{tikzpicture}
	\end{minipage}
	\hspace{0.5cm}
	\begin{minipage}{0.5 \textwidth}
\tikzset{mark size=1}\begin{tikzpicture}\begin{axis}[grid=major, xmin=0, ymin=0, xmax=30, ymax=33,ytick={0,3,12,13,15,16,17,20,24,26,27}, xtick={0,4,7,11,14,16,21,25}, yticklabel style={font=\tiny}, xticklabel style={font=\tiny}]\addplot[->, style=dotted, very thick]coordinates{(25.45454545454545,12)(30,12)};\addplot[->, style=dotted, very thick]coordinates{(25.45454545454545,15)(30,15)};\addplot[->, style=dotted, very thick]coordinates{(25.45454545454545,16)(30,16)};\addplot[->, style=dotted, very thick]coordinates{(25.45454545454545,18)(30,18)};\addplot[->, style=dotted, very thick]coordinates{(25.45454545454545,19)(30,19)};\addplot[->, style=dotted, very thick]coordinates{(25.45454545454545,21)(30,21)};\addplot[->, style=dotted, very thick]coordinates{(25.45454545454545,22)(30,22)};\addplot[->, style=dotted, very thick]coordinates{(25.45454545454545,24)(30,24)};\addplot[->, style=dotted, very thick]coordinates{(25.45454545454545,25)(30,25)};\addplot[->, style=dotted, very thick]coordinates{(25.45454545454545,26)(30,26)};\addplot[->, style=dotted, very thick]coordinates{(25.45454545454545,27)(30,27)};\addplot[->, style=dotted, very thick]coordinates{(14,27.45762711864407)(14,33)};\addplot[->, style=dotted, very thick]coordinates{(15,27.45762711864407)(15,33)};\addplot[->, style=dotted, very thick]coordinates{(18,27.45762711864407)(18,33)};\addplot[->, style=dotted, very thick]coordinates{(19,27.45762711864407)(19,33)};\addplot[->, style=dotted, very thick]coordinates{(21,27.45762711864407)(21,33)};\addplot[->, style=dotted, very thick]coordinates{(22,27.45762711864407)(22,33)};\addplot[->, style=dotted, very thick]coordinates{(23,27.45762711864407)(23,33)};\addplot[->, style=dotted, very thick]coordinates{(24,27.45762711864407)(24,33)};\addplot[->, style=dotted, very thick]coordinates{(25,27.45762711864407)(25,33)};\addplot [pattern = north east lines, draw=white]coordinates{(25,27)(30,27)(30,33)(25,33)(25,27)};\addplot[only marks] coordinates{(0,0)(11,17)(14,20)(14,23)(14,27)(15,21)(15,22)(15,23)(15,24)(15,25)(15,26)(15,27)(16,15)(16,16)(16,18)(16,19)(16,20)(20,16)(24,27)};\addplot[only marks,mark=o, mark options={scale=2.7}] coordinates{(4,3)(7,13)(11,17)(14,27)(15,27)(16,20)(24,27)(25,12)(25,16)(25,26)};\addplot[only marks, mark=text, mark options={scale=1,text mark={\scalebox{.5}{$\eta$}}},
	text mark as node=true] coordinates{(8,6)(11,9)(11,12)(11,15)(11,16)(12,9)(14,12)(14,15)(14,18)(14,19)(14,21)(14,22)(14,24)(14,25)(14,26)(15,12)(15,15)(15,18)(15,19)(16,12)(18,15)(18,18)(18,19)(18,21)(18,22)(18,24)(18,25)(18,26)(18,27)(19,15)(19,18)(19,19)(19,21)(19,22)(20,15)(21,18)(21,21)(21,22)(21,24)(21,25)(21,26)(21,27)(22,15)(22,18)(22,19)(22,21)(22,22)(22,24)(22,25)(22,26)(22,27)(23,15)(23,18)(23,19)(23,21)(23,22)(23,24)(23,25)(24,15)(24,18)(25,15)(25,18)(25,19)(25,21)(25,22)(25,24)(25,25)(25,27)};\addplot[only marks, mark=text, mark options={scale=1,text mark={\scalebox{.5}{$\eta$}}}, text mark as node=true] coordinates{(4,3)(7,6)(7,9)(7,12)(7,13)(14,16)(15,16)(18,12)(18,16)(19,12)(19,16)(20,12)(21,12)(21,15)(21,16)(21,19)(22,12)(22,16)(23,12)(23,16)(24,12)(24,16)(25,12)(25,16)(25,26)};\addplot[only marks] coordinates{(15,20)(18,23)(19,23)(19,24)(19,25)(19,26)(19,27)(20,18)(20,19)(20,21)(20,22)(20,23)(23,26)(23,27)(24,19)(24,21)(24,22)(24,24)(24,25)(24,26)};\end{axis}\end{tikzpicture}

	\end{minipage}
	\caption{$\circ$: Reducible elements; $\bigcirc$: Irreducible Absolutes; $\eta$: Elements of $\sgen{\operatorname{red}(\eta)}$}
\label{algfig2}
\end{figure}
Thus for this semigroup $\operatorname{bedim}(S)=5.$
We consider $\eta=\eta_1=\{ ( 4, 3 ), ( 7, 13 ), ( \infty, 12 ), ( \infty, 16 ),$ $( \infty, 26 ) \}$ and we want to show that $\eta \in \mathfrak{R}(S)$.

We have $ \eta^1=\red{\eta}=\{ ( 4, 3 ), ( 7, 13 ), (11,17), (14,\infty), (16,20), (24,\infty), ( \infty, 12 ), ( \infty, 16 ),$ $( \infty, 26 ) \}$.

In fact 
\begin{itemize} 
\item $\bs{\alpha_1}=(11,17)$ is reducible by $\eta$ because we have
 ${}_1\Delta^{\eta}(\bs{\alpha_1}) \neq \emptyset$ since $(4,3)\odot(7,13)=(11,16) \in \sgen{\eta}$. Furthermore ${}_1\delta^{\eta}(\bs{\alpha_1})=16$.

Since ${}_1\delta^{S}(\bs{\alpha_1})=16$ we need only to find an element of the type $(x,16) \in \sgen{\eta}$ with $x>11$. The element $(\infty,16) \in \eta$ satisfies this property.
\item $\bs{\alpha_2}=(14,\infty)$ is reducible by $\eta$ because we have	 ${}_1\Delta^{\eta}(\bs{\alpha_2}) \neq \emptyset$; in fact we have $2(7,13)=(14,26) \in \sgen{\eta}$. Furthermore ${}_1\delta^{\eta}(\bs{\alpha_2})=26$.

 Since $\tilde{y}=18$, for all 
 \begin{eqnarray*}
 y \in Y^{\eta}(\bs{\alpha_2})&=&\{y \in \{{}_1\delta^{\eta}(\bs{\alpha_2})=26,\ldots,\max\{\tilde{y},{}_1\delta^{\eta}(\bs{\alpha_2})\}+e_2-1=28| (14,y) \in S\}=\\
 &=&\{ 26,27,28\},
 \end{eqnarray*} we need  to find an element of the type $(x,y) \in \sgen{\eta}$ with $x>14$. The following elements of $\sgen{\eta}$ satisfy this property:
	$$ (\infty,26), \quad 9(4,3)=(36,27), \quad 5(4,3)\odot(7,13)=(27,28). 
	$$
	\item $\bs{\alpha_3}=(16,20)$ is reducible by $\eta$. In fact we have ${}_1\Delta^{\eta}(\bs{\alpha_3}) \neq \emptyset$; since $4(4,3)=(16,12) \in \sgen{\eta}$. Furthermore ${}_1\delta^{\eta}(\bs{\alpha_3})=12$.

 Since ${}_1\delta^{S}(\bs{\alpha_3})=19$, for all $y \in Y^{\eta}(\bs{\alpha_3})=\{12,15,16,18,19\}$ we need  to find an element of the type $(x,y) \in \sgen{\eta}$ with $x>16$.  The following elements of $\sgen{\eta}$ satisfy this property:
	$$ (\infty,12), \quad 5(4,3)=(20,15), \quad (\infty,16), \quad 6(4,3)=(24,18), \quad  (4,3)\odot (\infty,16)=(\infty,19). $$
	\item $\bs{\alpha_4}=(24,\infty)$ is reducible by $\eta$. In fact ${}_1\Delta^{\eta}(\bs{\alpha_4})\neq \emptyset$ since $6(4,3)=(24,18) \in \sgen{\eta}$. Thus ${}_1\delta^{\eta}(\bs{\alpha_4})=18$.
 Since $\tilde{y}=24$, for all 
 \begin{eqnarray*} y \in Y^{\eta}(\bs{\alpha_4})&=&\{y \in \{{}_1\delta^{\eta}(\bs{\alpha_4})=18,\ldots,\max\{\tilde{y},{}_1\delta^{\eta}(\bs{\alpha_4})\}+e_2-1=26| (24,y) \in S\}=\\
 	&=&\{18, 19, 21, 22, 24, 25, 26\},
 \end{eqnarray*}
 	 we need  to find an element of the type $(x,y) \in \sgen{\eta}$ with $x>24$. The following elements of $\sgen{\eta}$ satisfy this property:
	\begin{eqnarray*} 2(4,3)\odot (\infty,12)&=&(\infty,18) \quad (4,3) \odot (\infty,16)=(\infty,19), \quad 3(4,3) \odot (\infty,12)=(\infty,21),  \\2(4,3)\odot(\infty,16)&=&(\infty,22), \quad 2(\infty,12)=(\infty,24), \quad (7,13)\odot (\infty,12)=(\infty,25), \quad (\infty,26) . \end{eqnarray*}
\end{itemize}

Notice that $\bs{\alpha_5}=(15,\infty)$ is not reducible by $\eta$, but it is reducible by $\eta^1$. 
In fact ${}_1\Delta^{\eta^1}(\bs{\alpha_5}) \neq \emptyset$ since $(4,3)\odot(11,17)=(15,20) \in \sgen{\eta^1}$. Thus ${}_1\delta^{\eta^1}(\bs{\alpha_5}) =20$.
	Since $\tilde{y}=18$, for all 
	\begin{eqnarray*}
		y \in Y^{\eta^1}(\bs{\alpha_5})&=&\{y \in \{{}_1\delta^{\eta^1}(\bs{\alpha_5})=20,\ldots,\max\{\tilde{y},{}_1\delta^{\eta^1}(\bs{\alpha_5})\}+e_2-1=22| (14,y) \in S\}=\\
		=\{ 20, 21, 22\}, 
		\end{eqnarray*}
	we need  to find an element of the type $(x,y) \in \sgen{\eta^1}$ with $x>15$. The following elements of $\sgen{\eta^1}$ satisfy this property:
	$$ (16,20), \quad 3(4,3)\odot(\infty,12)=(\infty,21), \quad 3(4,3)\odot(7,13)=(19,22). $$

Thus $\red{\eta^1}=I_A(S)$, and this means $\eta \in \mathfrak{R}(S)$ since $\operatorname{red}(\eta)=I_A(S).$
Hence $\operatorname{Bedim}(S) \leq 5=|\eta|.$
Since we have $$ 5=\operatorname{bedim}(S) \leq \operatorname{edim}(S) \leq \operatorname{Bedim}(S) \leq 5,$$
we can finally deduce that $\operatorname{edim}(S)=5$ and $\eta$ is an \emph{msor}.

It is possible to check that all the minimal hitting sets previously found satisfy the reducibility condition, thus they are all \emph{msor} for $S$.
\end{ex}
All the previous computations were realized implementing all the previous algorithms in GAP \cite{GAP4}.
	
\section{Properties of embedding dimension}
\label{section3}
\subsection{Relationship between embedding dimension of a ring and embedding dimension of its value semigroup}
\label{section31}
\begin{teo}
\label{ring}
	Let $S$ be a good semigroup of $\N^2$ such that there exists an algebroid curve $R$ with $v(R)=S$. Then $\operatorname{edim}(S) \geq \operatorname{edim}(R)$.
\end{teo}
\begin{proof}
	Let us consider an algebroid curve $R$ such that $v(R)=S$ and denote by $\varepsilon$ the embedding dimension of $S$.	Thus there exists $\eta \subset I_A(S) $, \emph{msor} of $S$, with $| \eta |=\varepsilon$. We want to prove $\operatorname{edim}(R) \leq \varepsilon.$

	We denote  by
	$$
	\eta=\{ \bs{\alpha_1},\ldots,\bs{\alpha_{\varepsilon}}\},$$
	and we want to show that it is possible to choose elements $\phi_1,\ldots,\phi_{\varepsilon}$ in $R$, such that: \begin{itemize} \item  $v(\phi_j)=\bs{\alpha_j}$ for each $j=1,\ldots,\varepsilon;$
	\item  $v(\K[\![\phi_1,\ldots,\phi_{\varepsilon}]\!])$ is a good semigroup. 
	\end{itemize}
	Denote by $R_1=\K[\![\phi_1,\ldots,\phi_{\varepsilon}]\!]$.
 By construction, for each choice of the elements $\phi_j$, the subsemigroup $v(R_1) \subseteq \mathbb{N}^2$ always satisfies the properties (G1) and (G3) of good semigroups, thus we need to guarantee the existence of a conductor.
	This can be done by forcing in $v$ the presence of vectors that fulfil the conditions of Proposition \ref{boundcond} (it is not difficult to do that by accordingly adding to the $\phi_i$ elements of $R$ with value greater than its conductor). 
	
	Now,  $$ \eta \subseteq v(R_1) \subseteq v(R)=S,$$
	and, since $\eta$ is a \emph{msor} of $S$ and $v(R_1)$ is a good semigroup, we have $v(R_1)=S$.
	Notice that $R_1 \subseteq R$ with $v(R)=v(R_1)$ implies that
	$R_1=R$.
	In fact, considered an element $r \in R$, there exists an element $r_1 \in R_1$ such that $v(r)=v(r_1)$. Thus we can fin a $k_1 \in \mathbb{K}$ such that $v(r-k_1r_1)$ is strictly greater than $v(r)$. We eventually find $k_j \in \K$ and $r_j \in R_1$ such that $v(r-\sum{k_jr_j}) \geq c(S)=c(v(R_1))$, implying that $r-\sum{k_jr_j} \in R_1$, and $r \in R_1$.
	
	Thus $ \operatorname{edim}(R)=\operatorname{edim}(R_1) \leq \varepsilon=\operatorname{edim}(S)$. \end{proof}

We want to show that the inequality can be strict and we want to analyze the cases when this happens.

\begin{ex}
\label{exring1}
Let us consider the ring $R\cong \K[\![(t^4,u^4),(t^6+t^9,u^6+u^7),(2t^{15}+t^{18},2u^{13}+u^{14})]\!]$.

We observe that $R=\K[\![(t^4,u^4),(t^6+t^9,u^6+u^7),(2t^{15}+t^{18 },2u^{13}+u^{14})]\!]=\K[\![(t^4,u^4),(t^6+t^9,u^6+u^7)]\!]$, in fact:
$$(2t^{15}+t^{18},2u^{13}+u^{14})=(t^6+t^9,u^6+u^7)^2-(t^4,u^4)^3.$$
We have that $\operatorname{edim}(R)=2$, but $\operatorname{edim}(v(R))=3$, since $M=\{(4,4),(6,6),(15,13)\}$ is the only hitting set of the semigroup $v(R)$ \\
This fact happens because in the ring $R$ the element of value $(15,13)$ is obtained by the sum of  the elements of value $(4,4)$ and $(6,6)$ because of a cancellation.\\
This situation cannot be controlled by the property (G3) of the good semigroups. This gap in embedding dimension can be justified by the fact that this piece of information is lost in the passage from the ring to the semigroup.
For this value semigroup it is possible to find a  ring, namely $T=\K[\![(t^4,u^4),(t^6,u^6),(t^{15},u^{13})]\!]$ with $v(T)=v(R)$, and such that $\operatorname{edim}(T)=\operatorname{edim}(v(T))$. This situation is not guaranteed to happen in general, as it is shown in the following example. \\
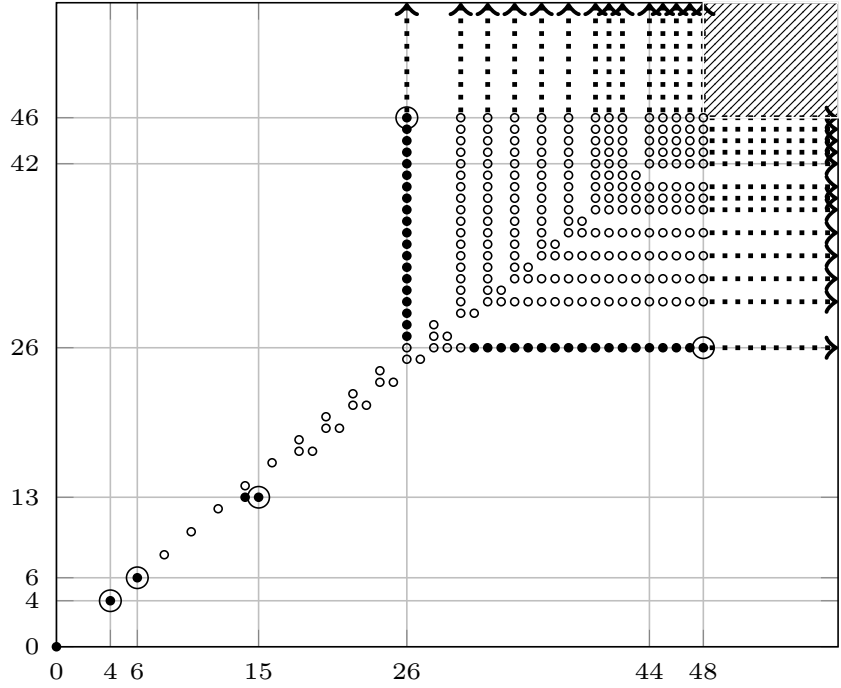
\begin{figure}
\centering
\tikzset{mark size=1}
\begin{tikzpicture}[scale=1.5]\begin{axis}[grid=major, xmin=0, ymin=0, xmax=58, ymax=56, ytick={0,4,6,13,26,42,46}, xtick={0,4,6,15,26,44,48}, yticklabel style={font=\tiny}, xticklabel style={font=\tiny}]\addplot[->, style=dotted, very thick]coordinates{(48.47524752475248,26)(58,26)};\addplot[->, style=dotted, very thick]coordinates{(48.47524752475248,30)(58,30)};\addplot[->, style=dotted, very thick]coordinates{(48.47524752475248,32)(58,32)};\addplot[->, style=dotted, very thick]coordinates{(48.47524752475248,34)(58,34)};\addplot[->, style=dotted, very thick]coordinates{(48.47524752475248,36)(58,36)};\addplot[->, style=dotted, very thick]coordinates{(48.47524752475248,38)(58,38)};\addplot[->, style=dotted, very thick]coordinates{(48.47524752475248,39)(58,39)};\addplot[->, style=dotted, very thick]coordinates{(48.47524752475248,40)(58,40)};\addplot[->, style=dotted, very thick]coordinates{(48.47524752475248,42)(58,42)};\addplot[->, style=dotted, very thick]coordinates{(48.47524752475248,43)(58,43)};\addplot[->, style=dotted, very thick]coordinates{(48.47524752475248,44)(58,44)};\addplot[->, style=dotted, very thick]coordinates{(48.47524752475248,45)(58,45)};\addplot[->, style=dotted, very thick]coordinates{(48.47524752475248,46)(58,46)};\addplot[->, style=dotted, very thick]coordinates{(26,46.47422680412371)(26,56)};\addplot[->, style=dotted, very thick]coordinates{(30,46.47422680412371)(30,56)};\addplot[->, style=dotted, very thick]coordinates{(32,46.47422680412371)(32,56)};\addplot[->, style=dotted, very thick]coordinates{(34,46.47422680412371)(34,56)};\addplot[->, style=dotted, very thick]coordinates{(36,46.47422680412371)(36,56)};\addplot[->, style=dotted, very thick]coordinates{(38,46.47422680412371)(38,56)};\addplot[->, style=dotted, very thick]coordinates{(40,46.47422680412371)(40,56)};\addplot[->, style=dotted, very thick]coordinates{(41,46.47422680412371)(41,56)};\addplot[->, style=dotted, very thick]coordinates{(42,46.47422680412371)(42,56)};\addplot[->, style=dotted, very thick]coordinates{(44,46.47422680412371)(44,56)};\addplot[->, style=dotted, very thick]coordinates{(45,46.47422680412371)(45,56)};\addplot[->, style=dotted, very thick]coordinates{(46,46.47422680412371)(46,56)};\addplot[->, style=dotted, very thick]coordinates{(47,46.47422680412371)(47,56)};\addplot[->, style=dotted, very thick]coordinates{(48,46.47422680412371)(48,56)};\addplot [pattern = north east lines, draw=white]coordinates{(48,46)(58,46)(58,56)(48,56)(48,46)};\addplot[only marks] coordinates{(0,0)(4,4)(6,6)(14,13)(15,13)(26,27)(26,28)(26,29)(26,30)(26,31)(26,32)(26,33)(26,34)(26,35)(26,36)(26,37)(26,38)(26,39)(26,40)(26,41)(26,42)(26,43)(26,44)(26,45)(26,46)(31,26)(32,26)(33,26)(34,26)(35,26)(36,26)(37,26)(38,26)(39,26)(40,26)(41,26)(42,26)(43,26)(44,26)(45,26)(46,26)(47,26)(48,26)};\addplot[only marks,mark=o, mark options={scale=2.7}] coordinates{(4,4)(6,6)(15,13)(26,46)(48,26)};\addplot[only marks, mark=o] coordinates{(8,8)(10,10)(12,12)(14,14)(16,16)(18,17)(18,18)(19,17)(20,19)(20,20)(21,19)(22,21)(22,22)(23,21)(24,23)(24,24)(25,23)(26,25)(26,26)(27,25)(28,26)(28,27)(28,28)(29,26)(29,27)(30,26)(30,29)(30,30)(30,31)(30,32)(30,33)(30,34)(30,35)(30,36)(30,37)(30,38)(30,39)(30,40)(30,41)(30,42)(30,43)(30,44)(30,45)(30,46)(31,29)(32,30)(32,31)(32,32)(32,33)(32,34)(32,35)(32,36)(32,37)(32,38)(32,39)(32,40)(32,41)(32,42)(32,43)(32,44)(32,45)(32,46)(33,30)(33,31)(34,30)(34,32)(34,33)(34,34)(34,35)(34,36)(34,37)(34,38)(34,39)(34,40)(34,41)(34,42)(34,43)(34,44)(34,45)(34,46)(35,30)(35,32)(35,33)(36,30)(36,32)(36,34)(36,35)(36,36)(36,37)(36,38)(36,39)(36,40)(36,41)(36,42)(36,43)(36,44)(36,45)(36,46)(37,30)(37,32)(37,34)(37,35)(38,30)(38,32)(38,34)(38,36)(38,37)(38,38)(38,39)(38,40)(38,41)(38,42)(38,43)(38,44)(38,45)(38,46)(39,30)(39,32)(39,34)(39,36)(39,37)(40,30)(40,32)(40,34)(40,36)(40,38)(40,39)(40,40)(40,41)(40,42)(40,43)(40,44)(40,45)(40,46)(41,30)(41,32)(41,34)(41,36)(41,38)(41,39)(41,40)(41,41)(41,42)(41,43)(41,44)(41,45)(41,46)(42,30)(42,32)(42,34)(42,36)(42,38)(42,39)(42,40)(42,41)(42,42)(42,43)(42,44)(42,45)(42,46)(43,30)(43,32)(43,34)(43,36)(43,38)(43,39)(43,40)(43,41)(44,30)(44,32)(44,34)(44,36)(44,38)(44,39)(44,40)(44,42)(44,43)(44,44)(44,45)(44,46)(45,30)(45,32)(45,34)(45,36)(45,38)(45,39)(45,40)(45,42)(45,43)(45,44)(45,45)(45,46)(46,30)(46,32)(46,34)(46,36)(46,38)(46,39)(46,40)(46,42)(46,43)(46,44)(46,45)(46,46)(47,30)(47,32)(47,34)(47,36)(47,38)(47,39)(47,40)(47,42)(47,43)(47,44)(47,45)(47,46)(48,30)(48,32)(48,34)(48,36)(48,38)(48,39)(48,40)(48,42)(48,43)(48,44)(48,45)(48,46)};\end{axis}\end{tikzpicture}
\caption{Semigroup $v(R)$ of the Example \ref{exring1}}
\end{figure}
\end{ex}

\begin{ex}
\label{exring2}
Let us consider the ring $R=[\![(t^4,u^3),(t^7,u^{13}),(t^{11},u^{17}),(t^{16},u^{20})]\!]$ that has embedding dimension $4$. Its value semigroup is the good semigroup that appeared in Example \ref{exalg}, where we proved that its embedding dimension is five. 

We focus on one of its \emph{msor}, namely  $\eta=\{(4,3),(7,13),(11,17),(16,20),(\infty,26)\}$.
If we analyze in detail what happens, we oberve that $(t^{23},u^{33})=(t^7,u^{13})\cdot(t^{16},u^{20})\in R$ and $(t^{23},u^{26})=(t^{11},u^{17})\cdot(t^{12},u^{9})\in R$, thus $(0,u^{26}-u^{33})\in R$.
But $\eta=\{(4,3),(7,13),(11,17),(16,20)\}$ is not a \emph{sor}, since we have seen in the Example \ref{exalg} that all MHS have to contain either $(\infty,26)$ or $(24,\infty)$. 
This fact happens because in the ring $R$ all the elements of value $(x,26)$ with $x\geq 25$ appear because we have a complete cancellation on the first component (i.e. we obtain $0$ on the first component).\\
In the semiring $\sgen{\eta}$ the existence of the elements $(23,33)$ and $(23,26)$ guarantees, by property (G3) of the good semigroups, only the existence of one element of value $(>23,26)$, but not the presence of all elements $(x,26)$, with $x\geq 24$.\\
 Also in this case in the semigroup we lose a piece of information present in the ring.

Differently from the previous example, it is not possible to find a ring $T$ such that $v(T)=v(R)$ and $\operatorname{edim}(T)=\operatorname{edim}(v(T))=5$.
To see this, let us suppose by contradiction that such a ring $T$ exists. Let us consider 
$\psi_1,\ldots,\psi_5 \in T$, such that \begin{itemize}
    \item $v(\psi_1)=(4,3)$;
     \item $v(\psi_2)=(7,13)$;
      \item $v(\psi_3)=(11,17)$;
       \item $v(\psi_4)=(16,20)$;
        \item $v(\psi_5)=(\infty,26)$.
\end{itemize}
From the proof of Theorem \ref{ring} we have that $T \cong \K[[\psi_1,\psi_2,\psi_3,\psi_4, \psi_5]].$ 

Let us consider the ring $T'\cong \K[[\psi_1,\psi_2,\psi_3,\psi_4]]$. We must have $v(T') \subsetneq v(T)$ because otherwise $T=T'$, against the fact that $\operatorname{edim}(T)=5$.
Now we have that $\{(4,3),(7,13),(11,17),(16,20) \} \subseteq v(T')$ and it is not difficult to show that there exists only one good semigroup $D$ containing these vectors and contained in $v(T)$. The good semigroup $D$ is the one appeared in \cite[Example 2.16]{anal:unr} as the first example of a good semigroup that cannot be a value semigroup of a ring. Thus $v(T')=v(T)$ and we have a contradiction.
\end{ex}

\subsection{Relationship between embedding dimension and multiplicity}
\label{section32}

Now we want to prove the following theorem.
\begin{teo} \label{mult}
	Let $S$ be a good semigroup. Denote by $\bs{e}=(e_1,e_2)$ the multiplicity vector of $S$. Then $\operatorname{edim}(S) \leq e_1+e_2$.
\end{teo}

We recall that, if $S$ is a numerical semigroup with multiplicity $e(S)$, it is possible to prove that $\operatorname{edim}(S)\leq e(S)$   using the fact that the  set $\operatorname{Ap}(S)\setminus \{ 0 \} \cup \{ e(S)\}$ is a system of generators of $S$ with cardinality $e(S).$ 
Using the properties of the Apéry set of a good semigroup, introduced in \cite{DAGuMi},  we wish to prove the same inequality for good semigroups contained in $\N^2$.\\
First of all, we recall the notion of Apéry set and levels. 
\begin{defi}
The Apéry set of the  good semigroup $S$ (with respect to the multiplicity) is defined as the set:
$$\operatorname{Ap}(S)=\{\bs{\alpha}\in S: \bs{\alpha}-\bs{e} \notin S \}.$$
\end{defi}

 We say that $(\alpha_1,\alpha_2)\leq\leq (\beta_1,\beta_2)$ if and only if $(\alpha_1,\alpha_2)=(\beta_1,\beta_2)$ or $(\alpha_1,\alpha_2)\neq (\beta_1,\beta_2)$ and we have
  $(\alpha_1,\alpha_2)\ll (\beta_1,\beta_2)$ where the last means $\alpha_1<\beta_1$ and $\alpha_2<\beta_2$.\\
As described in \cite{DAGuMi}, it is possible to build up a partition of the Apéry set, in the following way. Let us define, $D^{0}=\emptyset$:
$$B^{(i)}=\{\bs{\alpha}\in \operatorname{Ap}(S)\backslash (\cup_{j<i}D^{(j)}): \bs{\alpha} \text{ is maximal with respect to } \leq \leq\}$$
$$C^{(i)}=\{\bs{\alpha}\in B^{(i)}: \bs{\alpha}=\bs{\beta}_1\oplus \bs{\beta}_2 \text{ for some } \bs{\beta}_1,\bs{\beta}_2\in B^{(i)}\setminus\{\bs{\alpha}\}\}$$
$$D^{(i)}=B^{(i)}\backslash C^{(i)}.$$
For a certain $N\in \N$, we have $\operatorname{Ap}(S)=\cup_{i=1}^ND^{(i)}$ and $D^{(i)}\cap D^{(j)}=\emptyset$. In according to notation of \cite{DAGuMi}, we rename these sets in an increasing order setting $A_i=D^{(N+1-i)}.$ Thus we have
$$\operatorname{Ap}(S)=\cup_{i=1}^NA_i.$$
Notice that the first level $A_1$ of $\operatorname{Ap}(S)$ consists only of the zero vector.
It was proved [Thm. 3.4 \cite{DAGuMi}] that $N=e_1+e_2$, a key result in the proof of our inequality.

In order to simplify the notation in the following results we define the  set $\overline{\operatorname{Ap}(S)}=\left(\operatorname{Ap}(S)\setminus\{\bs{0}\}\right) \cup \{\bs{e} \}$.
Since we are only interchanging the role of the multiplicity vector and the zero vector, we have 
$$\overline{\operatorname{Ap}(S)}=\cup_{i=1}^{N}A'_i,$$
where $A_i=A'_i$ for $i=2,\ldots,N$, and $A'_1=\{\bs{e}\}.$

In order to prove Theorem \ref{mult}, it is useful to introduce the following new definition of reducibility of an element of $I_A(S)$ by a subset $\eta \subseteq I_A(S)$.

\begin{defi} \label{rho}
	Let $\bs{\alpha}=(\alpha_1,\alpha_2) \in I_A(S)$ and $\eta \subseteq I_A(S)$.
\begin{itemize}
	\item \textbf{Case} $(\alpha_1,\alpha_2) \in I_{A_f}(S)$. 
	Then $\bs{\alpha}$ is $\rho$-reducible by $\eta$ if 
	\begin{enumerate}
		\item $\exists \bs{h}_1,\ldots,\bs{h}_k \in \eta$ such that
		$\bs{h}_1 \odot \dots \odot \bs{h}_k=(\beta_1,\alpha_2)$ with $\beta_1<\alpha_1$.
		
		\item $\forall x\in \{\beta_1,\ldots,{}_2\delta^{S}(\bs{\alpha})\}$ such that $ (x,\alpha_2)\in S $ we can find $\bs{j}_1,\ldots,\bs{j}_l \in \eta$ such that $\bs{j}_1 \odot \dots \odot \bs{j}_l=(x,\beta_2)$  with $\beta_2>\alpha_2$.
		\end{enumerate}
			\item \textbf{Case} $\bs{\alpha}=(\infty,\alpha_2) \in I_{A}(S)^{\infty}$. Denote, as we did before, by $\tilde{x}$ the minimal element such that $(x,\alpha_2) \in S$ for all $x\geq \tilde{x}$.
		Then $(\infty,\alpha_2)$ is $\rho$-reducible by $\eta$ if 
		\begin{enumerate}
			\item $\exists \bs{h}_1,\ldots,\bs{h}_k \in \eta$ such that
			$\bs{h}_1 \odot \dots \odot \bs{h}_k=(\beta_1,\alpha_2)$ with $\beta_1<\infty$.
			
			\item $\forall \tilde{x} \in \{ x\in \{\beta_1,\ldots,\max(\beta_1,\tilde{x})+e_1-1\}: (x,\alpha_2)\in S \}$ we can find $\bs{j}_1,\ldots,\bs{j}_l \in \eta$ such that $\bs{j}_1 \odot \dots \odot \bs{j}_l=(\tilde{x},\beta_2)$  with $\beta_2>\alpha_2$.
	\end{enumerate}
	\item \textbf{Case} $\bs{\alpha}=(\alpha_1,\infty) \in I_{A}(S)^{\infty}$. Such an element is never $\rho$-reducible by $\eta$. 
\end{itemize}
\end{defi}
\begin{oss}
	If an element of $I_A(S)$ is $\rho$-reducible by $\eta$, it is also reducible by $\eta$.
\end{oss}	
\begin{oss} \label{red}
	If an element $(\alpha_1,\alpha_2)$ of $I_A(S)$ is $\rho$-reducible by $\eta$, then it is also $\rho$-reducible by $\eta_{\alpha_1}=\{ (x,y) \in \eta: x<\alpha_1 \}$.
	In fact, the elements required to satisfy the condition 1) and 2) of Definition \ref{rho} cannot be obtained by using irreducible absolute elements of $S$ with first component bigger than $\alpha_1$ (because we only allow the operation $\odot$  to produce them). 
\end{oss}

Now we write $$ I_A(S)=\{ \bs{\alpha}^{(1)}=(\alpha_1^{(1)},\alpha_2^{(1)}),\ldots,\bs{\alpha}^{(n)}=(\alpha_1^{(n)},\alpha_2^{(n)}) \}, $$ where the elements are ordered in decreasing order with respect to the first coordinate, i.e. if $j<l$, then $ \alpha_1^{(j)}>\alpha_1^{(l)}$ or $\alpha_1^{(j)}=\alpha_1^{(l)}=\infty$ and $\alpha_2^{(j)}>\alpha_2^{(l)}$.
Let us consider the following algorithm to produce, starting from $I_A(S)$, a set $\eta$ that is still a \emph{sor}  for $S$.

\begin{algorithm}
	\SetKwData{Left}{left}
	\SetKwData{This}{this}
	\SetKwData{Up}{up}
	\SetKwFunction{Union}{Union}
	\SetKwFunction{FindCompress}{FindCompress}
	\SetKwInOut{Input}{input}
	\SetKwInOut{Output}{output}
	\Input{The set of irreducible absolute elements $I_A(S)$}
	\Output{A subset $\eta \subseteq I_A(S)$}
	\BlankLine
	$\eta \longleftarrow I_A(S)$
	
	\For{$k\leftarrow 1$ \KwTo $n$} {\If{$\bs{\alpha}^{(k)} \textrm{ is } \rho\textrm{-reducible by } I_A(S)\setminus\{\bs{\alpha}^{(k)}\}$ }{$\eta \longleftarrow \eta\setminus\{\bs{\alpha}^{(k)}\}$}}
	\Return $ \eta $
	\caption{A way to produce a \emph{sor} using $\rho$-reducibility}
	\label{algo}
\end{algorithm}
\begin{prop}
    The output $\eta$ of Algorithm \ref{algo} is a \emph{sor} for $S$
    \end{prop}
\begin{proof}

Let us prove by induction on $k$ that the subset $\eta$ produced by the algorithm is a \emph{sor} for $S$. By Theorem \ref{E21}, we can do it by showing that it satisfies the reduciblity condition.
At the first step $\eta=I_A(S)$, hence we have a \emph{sor} for $S$. Suppose that at the $k$th step of the algorithm $\eta \in \mathfrak{R}(S)$ and let us show that it still satisfies the reducibility condition  after the $k+1$th step.
If $\bs{\alpha}^{(k+1)}$ is not $\rho$-reducible by $ I_A(S)\setminus\{\bs{\alpha}^{(k+1)}\}$, then we have nothing to prove since $\eta$ remains unchanged.
Now let us suppose that $\bs{\alpha}^{(k+1)}$ is $\rho$-reducible by $ I_A(S)\setminus\{\bs{\alpha}^{(k+1)}\}$. We need to prove that $\eta \setminus \{ 
\bs{\alpha}^{(k+1)} \}=\eta' \in \mathfrak{R}(S)$.
By Remark \ref{red}, $\bs{\alpha}^{(k+1)}$ is $\rho$-reducible by the set $W=\{ (\alpha_1,\alpha_2) \in I_A(S)\setminus\{\bs{\alpha}^{(k+1)}\}: \alpha_1<\alpha_1^{(k+1)} \}=\{ \bs{\alpha}^{(k+2)}, \ldots,\bs{\alpha}^{(n)}\}$.
But at this step of the algorithm $W \subseteq \eta'$, thus $\bs{\alpha}^{(k+1)}$ is $\rho$-reducible by $\eta'$, thus also reducible by $\eta$ and this means that $\eta \subseteq \operatorname{red}(\eta')$. By the inductive step $I_A(S)=\operatorname{red}(\eta) \subseteq \operatorname{red}(\eta')$, hence $\eta' \in \mathfrak{R}(S)$ and it is still a \emph{sor}.
    
\end{proof}

\begin{prop} \label{prop}
	If $\bs{\alpha}=(\alpha_1,\alpha_2) \in I_A(S)$ is such that 
	$_2\Delta^S(\bs{\alpha}) \not \subseteq \overline{\operatorname{Ap}(S)}$, then  $\bs{\alpha}$ is $\rho$-reducible by $I_A(S)\setminus\{\bs{\alpha}\}$.
\end{prop}
\begin{proof}
	Let us choose $(\beta_1,\alpha_2) \notin \overline{\operatorname{Ap}(S)}$ with the largest possible $\beta_1$.
	Thus, there exists an integer $k\geq 1$ such that $(\tilde{\alpha}_1,\tilde{\alpha}_2) \odot k(e_1,e_2)  =(\beta_1,\alpha_2)$, where $ (\tilde{\alpha}_1,\tilde{\alpha}_2)\in \overline{\operatorname{Ap}(S)} \cup \{\bs{0}\}$.
	Notice that, if $(\tilde{\alpha}_1,\tilde{\alpha}_2)=\bs{0}$, then $k\geq2$, otherwise we would have  $(\beta_1,\alpha_2)=(e_1,e_2) \in\overline{\operatorname{Ap}(S)}$.

If  $(\tilde{\alpha}_1,\tilde{\alpha}_2)\neq \bs{0}$, we write it as
	$$(\tilde{\alpha}_1,\tilde{\alpha}_2)=\bs{h}_1 \odot \dots \odot \bs{h}_l,$$
	where the $\bs{h}_j$ are irreducible elements of $S$.

Each $\bs{h}_j=	 (\alpha_1^j, \alpha_2^j)$ is an absolute element.  In fact, if it were possible to write it as $$(x,\alpha_2^j) \oplus (\alpha_1^j,y), \textrm{ with  } x>\alpha_1^j \textrm{ and } y>\alpha_2^j,$$ and  $(x,\alpha_2^j), (\alpha_1^j,y) \in S$, then it would follow that
	$$ \bs{h}_1 \odot \dots \odot (x,\alpha_2^j) \odot \dots \odot \bs{h}_l\odot k(e_1,e_2)=(\gamma_1,\alpha_2) \notin \overline{\operatorname{Ap}(S)}, $$
	and $\gamma_1>\beta_1$, this is against the maximality of $\beta_1$.
	
	Thus $\bs{h}_i \in I_A(S)$ for all $i$ (and they are clearly distinct from $(\alpha_1,\alpha_2))$.
	
	Now, if $(e_1,e_2) \in I_A(S)$, then $$ (\beta_1,\alpha_2)=k(e_1,e_2) \odot \bs{h}_1 \odot \dots \odot \bs{h}_l $$
	is already the element required to fulfill condition 1. in Definition \ref{rho}.
	
	Thus, let us suppose that $(e_1,e_2)=(\tilde{e_1},e_2) \oplus
	(e_1,\tilde{e_2
})$, where $\tilde{e_1}>e_1$, $\tilde{e_2}>e_2$ and $(\tilde{e_1},e_2), (e_1,\tilde{e_2}) \in I_A(S) \setminus\{ (\alpha_1,\alpha_2)\}$.
Notice that $\bs{\alpha}$ cannot be of the type $(\tilde{e_1},e_2)$ or $(e_1,\tilde{e_2}) $ because in both cases we would have 
$_2\Delta^S(\bs{\alpha}) \subseteq\overline{\operatorname{Ap}(S)}$ against our hypothesis.

First of all notice that $\tilde{e_1} \neq \infty$. In fact, if it were equal to $\infty$, then there would exist $\overline{x}$ such that $(x,e_2)\in S $ for all $x \geq \overline{x}$.
This implies that
$$ k(x,e_2) \odot \bs{h}_1 \odot \dots \odot \bs{h}_l =(kx+\tilde{\alpha_1},\alpha_2) \in S $$
for all $x \geq \overline{x}$. Thus $(\alpha_1,\alpha_2)=(\infty,\alpha_2)$ and this is a contradiction since
$$ (\alpha_1,\alpha_2)=(\infty,\alpha_2)=k(\infty,e_2)\odot \bs{h}_1 \odot \dots \odot \bs{h}_l, $$
is not an element of $I_A(S)$ being reducible (recall that if $\bs{h}_1 \odot \dots \odot \bs{h}_l=\bs{0}$, then $k\geq2$). 
Thus $\tilde{e_1} \neq \infty$, and the element 
$$ (\overline{\alpha}_1,\alpha_2)=k(\tilde{e_1},e_2)\odot \bs{h}_1 \odot \dots \odot \bs{h}_l, $$
is the required element that satisfies the condition 1. of Definition \ref{rho}.

Now we want to show that we can satisfy the condition 2. of $\rho$-reducibility.
Let us suppose that $\bs{\alpha}=(\alpha_1,\alpha_2) \in I_{A_f}(S)$ (all the following considerations can be adapted to the case $(\alpha_1,\alpha_2)=( \infty,\alpha_2)$).

We have to show that for each $\tilde{x} \in  X=\{ x\in \{\beta_1,\ldots,{}_2\delta^{S}(\bs{\alpha})\}: (x,\alpha_2)\in S \}$ we can find $\bs{j}_1,\ldots,\bs{j}_l \in \eta$ such that $\bs{j}_1 \odot \dots \odot \bs{j}_l=(\tilde{x},\beta_2)$  with $\beta_2>\alpha_2$.

Thus, let us consider an arbitrary $\tilde{x} \in X$.
Since $(\tilde{x},\alpha_2),(\alpha_1,\alpha_2) \in S$, by the (G3) property of Definition \ref{E1}, there exists $\beta_2 >\alpha_2 $ such that $(\tilde{x},\beta_2) \in S$. 

Theorem \ref{E7} ensures that we can write
$$(\tilde{x},\beta_2)=\bigoplus_{i=1}^m(\bigodot_{j=1}^n\bs{\gamma}_{j_i}), \bs{\gamma}_{j_i}\in I_A(S). $$
It must exist an index $\overline{i}$ such that
$$\bigodot_{j=1}^n\bs{\gamma}_{j_{\overline{i}}}=(\tilde{x},\tilde{\beta}_2).$$
Notice that $\bs{\gamma}_{j_{\overline{i}}} \in I_A(S) \setminus \{ (\alpha_1,\alpha_2)\}$ for all $j=1,\ldots,n$ (they all have first coordinate less than $\tilde{x}\leq \alpha_1$).
Furthermore $\tilde{\beta}_2 \geq \beta_2>\alpha_2$, thus it is the element which we were looking for in order to satisfy the condition 2. of $\rho$-reducibility.

\end{proof}

As a consequence of Proposition \ref{prop} and Algorithm \ref{algo}, we can immediately deduce the following Corollary.
\begin{cor} \label{cor}
	Let $S$ be a good semigroup.
Then the set $$\eta_S=\{ \bs{\alpha} \in I_A(S):\quad _2\Delta^S(\bs{\alpha}) \subseteq \overline{\operatorname{Ap}(S)}\}$$
is a \emph{sor} for $S$.
	\end{cor}

Now we are ready to give a proof of Theorem \ref{mult}.
\begin{proof}[Proof of Theorem \ref{mult}]
	Using Corollary \ref{cor} and the definition of embedding dimension, it suffices to show that $|\eta_S| \leq e_1+e_2$.
	
	Let us write $\eta_S=\{ \bs{h}^{(1)}=(\alpha_1^{(1)},\alpha_2^{(1)}),\ldots,\bs{h}^{(k)}=(\alpha_1^{(k)},\alpha_2^{(k)})\}$ where  if $i<j$ then $\alpha_2^{(i)}<\alpha_2^{(j)}$ or  $\alpha_2^{(i)}=\alpha_2^{(j)}=\infty$ with  $\alpha_1^{(i)}<\alpha_2^{(j)}$.
	Furthermore we denote by $c=(c_1,c_2)$ the conductor of $S$.
	Now to each element $\bs{h}^{(i)}$ of $\eta_S$ we associate an element $\overline{\bs{h}}^{(i)}$ in the following way:
	\begin{itemize}
		\item \textbf{Case} $\bs{h}^{(i)}=(\alpha_1,\infty)$. Then we set 
		$\overline{\bs{h}}^{(i)}=(\alpha_1,c_2+i).$
		\item  \textbf{Case} $\bs{h}^{(i)}=(\alpha_1,\alpha_2)$, with $\alpha_2 \neq \infty$.
Then we set $\overline{\bs{h}}^{(i)}=\min(_2\Delta^S(\bs{h}^{(i)})\cup \{\bs{h}^{(i)} \}) .$	
\end{itemize}
We consider the set $\eta'=\{ \overline{\bs{h}}^{(1)},\ldots,\overline{\bs{h}}^{(k)}\}$, and we want to show that distinct elements of $\eta'$ belong to distinct levels of the Apéry set of $S$.
In order to do that we consider two arbitrary elements $\overline{\bs{h}}^{(i)}$ and $\overline{\bs{h}}^{(j)}$ of $\eta'$ and we prove that they cannot belong to the same level of the Apéry set.
We have four possible configurations:
\begin{itemize}
	\item \textbf{Case} $\overline{\bs{h}}^{(i)}=(\alpha_1^{(i)},\alpha_2^{(i)})$ and $\overline{\bs{h}}^{(j)}=(\alpha_1^{(j)},\alpha_2^{(j)})$, with
	$\alpha_1^{(i)}<\alpha_1^{(j)}$ and $\alpha_2^{(i)}<\alpha_2^{(j)}$. 
	
	In this case $\overline{\bs{h}}^{(i)}\ll \overline{\bs{h}}^{(j)}$ and from definition of Apéry levels it follows that $\overline{\bs{h}}^{(j)} \in A_n$ and $\overline{\bs{h}}^{(i)} \in A_m$ with $m<n$.

	\item \textbf{Case}  $\overline{\bs{h}}^{(i)}=(\alpha_1^{(i)},\alpha_2^{(i)})$ and $\overline{\bs{h}}^{(j)}=(\alpha_1^{(j)},\alpha_2^{(j)})$, with
	$\alpha_1^{(i)}<\alpha_1^{(j)}$ and $\alpha_2^{(i)}=\alpha_2^{(j)}$. 

This configuration is not possible, because it is against the minimality of the element $\overline{\bs{h}}^{(j)}$ (it is easy to check that this situation cannot involve elements that come from $\bs{h}^{(i)}$ of the type $(\alpha_1,\infty)$).
	\item \textbf{Case}  $\overline{\bs{h}}^{(i)}=(\alpha_1^{(i)},\alpha_2^{(i)})$ and $\overline{\bs{h}}^{(j)}=(\alpha_1^{(j)},\alpha_2^{(j)})$, with
	$\alpha_1^{(i)}<\alpha_1^{(j)}$ and $\alpha_2^{(i)}>\alpha_2^{(j)}$. 

This configuration is not possible, since the element $\overline{\bs{h}}^{(i)} \oplus \overline{\bs{h}}^{(j)} \in S$  is against the minimality of the element $\overline{\bs{h}}^{(j)}$ (it is also easy to check that this situation cannot involve elements that come from $\bs{h}^{(i)}$ of the type $(\alpha_1,\infty)$).
	\item \textbf{Case}  $\overline{\bs{h}}^{(i)}=(\alpha_1^{(i)},\alpha_2^{(i)})$ and $\overline{\bs{h}}^{(j)}=(\alpha_1^{(j)},\alpha_2^{(j)})$, with
	$\alpha_1^{(i)}=\alpha_1^{(j)}$ and $\alpha_2^{(i)}>\alpha_2^{(j)}$. 
Suppose by contradiction that there exists $n \in \mathbb{N}$ such that $\overline{\bs{h}}^{(i)},\overline{\bs{h}}^{(j)}\in A_n$.
From the definition of $\eta_S$ it follows that $\Delta_2^S(\overline{\bs{h}}^{(j)}) \subseteq \overline{\operatorname{Ap}(S)}$.
Thus from Lemma 3.3 (3) of \cite{DAGuMi}, the minimal element $\bs{\beta}$ of $\Delta_2^S(\overline{\bs{h}}^{(j)}) \in A_m$ with $m\leq n$.
On the other hand $\overline{\bs{h}}^{(j)} \leq \bs{\beta}$, thus $\bs{\beta} \in A_l$ with $l\geq n$.
Thus $ \bs{\beta} \in A_n $ and this is a contradiction because we have
$$ \overbrace{\overline{\bs{h}}^{(i)}}^{\in A_n} \oplus \overbrace{\bs{\beta}}^{\in A_n}= \overline{\bs{h}}^{(j)} \in A_n,$$
that is against the definition of Apéry set level.
Since Theorem 3.4 of \cite{DAGuMi}, states that the levels of the Apéry Set are exactly  $e_1+e_2$, it follows that
$$ \operatorname{edim}(S) \leq |\eta_S|=|\eta'| \leq e_1+e_2,$$
and the proof of Theorem \ref{mult} is complete.
\end{itemize}
\end{proof}
We recall that a good semigroup is said to be Arf if  and only if $S(\bs{\alpha})=\{\bs{\beta}\in S|\bs{\beta}\geq\bs{\alpha}\}$ is a semigroup for any $\bs{\alpha}\in S$. In \cite[Proposition 3.19 and Corollary 5.8]{anal:unr} the authors proved that an Arf semigroup can be always seen as the value semigroup of an Arf ring. From this result and Theorem \ref{mult} we can deduce the following corollary.
\begin{cor} \label{arf}
	Let $S$ be an Arf good subsemigroup of $\mathbb{N}^2$. Then, denoted as usual by $\bs{e}=(e_1,e_2)$  the multiplicity vector of $S$,  we have $\operatorname{edim}(S)=e_1+e_2$.
\end{cor}
\begin{proof}
	By Theorem  \ref{mult} we have $\operatorname{edim}(S) \leq e_1+e_2$.
	Denote by $R$ an Arf ring such that $v(R)=S$.
	By Theorem \ref{ring} we have $\operatorname{edim}(S) \geq \operatorname{edim}(R)$.
	But $R$ is an Arf ring, thus its embedding dimension is equal to its multiplicity (cf.\cite[Theorem 2.2]{Lipman:stable}). Since the multiplicity of $R$ is also equal to $e_1+e_2$, we have 
	$$ e_1+e_2=\operatorname{edim}(R) \leq \operatorname{edim}(S)\leq e_1+e_2,  $$
and the proof of the corollary is complete.
\end{proof}
We say that a good semigroup $S \subseteq \mathbb{N}^2$ is \emph{maximal embedding dimension } if $\operatorname{edim}(S)=e_1+e_2$. Thus, Arf good semigroups constitute a particular class of maximal embedding dimension semigroups.
It is known that a numerical semigroup is  maximal embedding dimension if and only if $M+M=e+M$ where $M=S \setminus \{0\}$ is its maximal ideal and $e$ is its multiplicity (cf.\cite{max:pedro}).

Thus, we propose the following conjecture.
\begin{conj} \label{cong}
Let $S$ be a good subsemigroup of $\mathbb{N}^2$. Then $S$ is maximal embedding dimension if and only if 
$M+M=\bs{e}+M$, where $\bs{e}$ is its multiplicity vector and $M=S\setminus\{\bs{0}\}$.
\end{conj}
At the moment we have tested Conjecture \ref{cong} for a large number of good semigroup, and we have a proof of the fact that $M+M=\bs{e}+M$ implies $\operatorname{edim}(S)=e_1+e_2.$
 \begin{acknowledgements}
	The authors would like to thank Marco D'Anna and Pedro A. García Sánchez for their helpful comments and suggestions during the development of this paper. They also thank the "INdAM" and the organizers of the "International meeting on numerical semigroups (Cortona 2018)" for inviting them to attend the conference, which was of great help in developing some of the ideas in this paper. Finally, a special thank goes to the anonymous referee for the interesting and extensive comments on an earlier version of this paper.
\end{acknowledgements}


\begin{thebibliography}{10}
	
	
	\bibitem{anal:unr}
	V.~{Barucci}, M.~{D'Anna}, and R.~{Fröberg}.
	\newblock Analytically unramified one-dimensional semilocal rings and their
	value semigroups.
	\newblock {\em Journal of Pure and Applied Algebra}, 147(3):215--254, 2000.
	
	\bibitem{two:danna}
	V.~{Barucci}, M.~{D'Anna}, and R.~{Fröberg}.
	\newblock The semigroup of values of a one-dimensional local ring with two
	minimal primes.
	\newblock {\em Comm. Algebra}, 28(8):3607--3633, 2000.
	
	\bibitem{Berge:Hypergraphs}
	C.~Berge.
	\newblock {\em Hypergraphs}, volume~45.
	
	
	\bibitem{gener:cdgz}
	A.~{Campillo}, F.~{Delgado}, and S.~{Gusein-Zade}.
	\newblock On generators of the semigroup of a plane curve singularity.
	\newblock {\em J. London Math. Soc.}, 2(60):420--430, 1999.
	
	\bibitem{symm:cdk}
	A.~{Campillo}, F.~{Delgado}, and K.~{Kiyek}.
	\newblock Gorenstein properties and symmetry for one-dimensional local
	cohen-macaulay rings.
	\newblock {\em Manuscripta Math.}, 83:405--423, 1994.
	
	\bibitem{Carvalho:Semiring}
	E.~{Carvalho} and M.~{Escudeiro Hernandes}.
	\newblock {The Semiring of Values of an Algebroid Curve}.
	\newblock {\em ArXiv e-prints}, April 2017.
	
	\bibitem{good:danna}
	M.~D'Anna, P.A. García-Sánchez, V.~Micale, and L.~Tozzo.
	\newblock Good subsemigroups of {$\mathbb{N}^n$}.
	\newblock {\em International Journal of Algebra and Computation},
	28(02):179--206, 2018.
	
	\bibitem{DAGuMi}
	M.~{D'Anna}, L.~{Guerrieri}, and V.~{Micale}.
	\newblock {The Apéry Set of a Good Semigroup}.
	\newblock {\em ArXiv e-prints}, December 2018.
	
	\bibitem{canonical:danna}
	M.~D’Anna.
	\newblock Canonical module of a one-dimensional reduced local ring.
	\newblock {\em Comm. Algebra}, 25(09):2939--2965, 1997.
	
	\bibitem{symm:delgado}
	F.~Delgado.
	\newblock Gorenstein curves and symmetry of the semigroup of value.
	\newblock {\em Manuscripta Math.}, 61:285--296, 1988.
	
	\bibitem{multibranch:delgado}
	F.~Delgado.
	\newblock The semigroup of values of a curve singularity with several branches.
	\newblock {\em Manuscripta Math.}, 59:347--374, 1987.
	
	
	\bibitem{Eiter:Hypergraph}
	T.~Eiter and G.~Gottlob.
	\newblock Identifying the minimal transversals of a hypergraph and related
	problems.
	\newblock {\em SIAM Journal on Computing}, 24(6):1278--1304, 1995.
	
	\bibitem{GAP4}
	The GAP~Group.
	\newblock {\em {GAP -- Groups, Algorithms, and Programming, Version 4.10.0}},
	2018.
	
	\bibitem{semig:garcia}
	A.~Garc\'ia.
	\newblock Gorenstein curves and symmetry of the semigroup of value.
	\newblock {\em J. Reine Angew. Math.}, 336:165--184, 1982.
	
	
	\bibitem{Lipman:stable}
	J.~Lipman.
	\newblock Stable ideals and arf ring.
	\newblock 93:649--685.
	
	\bibitem{Murakami:hs}
	K.~Murakami and T.~Uno.
	\newblock Efficient algorithms for dualizing large-scale hypergraphs.
	\newblock {\em Discrete Applied Mathematics}, 170:83 -- 94, 2014.
	
	\bibitem{max:pedro}
	J.C. {Rosales}, P.~A. {García-Sánchez}, J.~I. {García-García}, and M.~B.
	{Branco}.
	\newblock Numerical semigroups with maximal embedding dimension.
	\newblock {\em Int. J. Commutative Rings}, 2:47–53, 2003.
	
	\bibitem{gruppe:waldi}
	R.~Waldi.
	\newblock Wertehalbgruppe und singularität einer ebenen algebroiden kurve.
	\newblock 01 1972.
	
\end{thebibliography}
\end{document}